 
\title{Strong solutions to McKean--Vlasov SDEs associated to a class of degenerate Fokker--Planck equations with coefficients of Nemytskii-type}
\author{Sebastian Grube\thanks{
                  Faculty of Mathematics, Bielefeld University, 33615 Bielefeld, Germany. E-Mail: sgrube@math.uni-bielefeld.de}
        }
\documentclass{article}
\usepackage[a4paper, left=2.05cm, right=2.05cm, top=2cm]{geometry}
\usepackage[style=alphabetic, backend=biber,giveninits, doi=false,isbn=false,url=false,eprint=false, maxbibnames=20, date=year]{biblatex}
\renewbibmacro{in:}{}
\DeclareNameAlias{author}{last-first}
\usepackage{setspace}
\usepackage{graphicx}
\usepackage[utf8]{inputenc}
\usepackage{amsmath}
\usepackage{amsfonts}
\usepackage{amssymb}
\usepackage{amsthm}
\usepackage{geometry}
\usepackage{hyperref}
\usepackage{mathrsfs}
\usepackage{bbm}
\usepackage{esint}
\usepackage{enumerate}
\usepackage{enumitem}
\usepackage{bm}

\newcommand{\law}[1]{{\mathcal L}_{#1}}
\newcommand{\LN}{\mathbb N}
\newcommand{\RR}{\mathbb R}

\newcommand{\BBBB}{\mathcal B}

\newcommand{\EE}{\mathbb E}
\newcommand{\FF}{\mathcal F}

\newcommand{\PP}{\mathbb P}

\newcommand{\PPPP}{\mathcal P}
\newcommand{\SSS}{\mathscr S}

\newcommand{\dx}{\mathrm dx}
\newcommand{\dy}{\mathrm dy}
\newcommand{\dr}{\mathrm dr}
\newcommand{\ds}{\mathrm ds}
\newcommand{\inv}{^{-1}}
\newcommand{\rd}{{\RR^d}}

\DeclareMathOperator{\divv}{\mathrm{div}}
\DeclareMathOperator{\M}{\mathrm{M}}
\DeclareMathOperator{\Lip}{\mathrm{Lip}}

\usepackage{nicefrac}

\usepackage{xifthen}

\newcommand{\norm}[2][]{\ifthenelse{\isempty{#1}}
	{\left\lVert#2\right\rVert}
	{\left\lVert#2\right\rVert_{{#1}}}
}

\newtheorem{theorem}{Theorem}[section]
\newtheorem{lemma}[theorem]{Lemma}
\newtheorem{corollary}[theorem]{Corollary}
\newtheorem{definition}[theorem]{Definition}
\newtheorem{remark}[theorem]{Remark}
\newtheorem{proposition}[theorem]{Proposition}

\begin{document}
\newpage
\maketitle
\begin{abstract}
While the nondegenerate case is well-known, there are only few results on the existence of strong solutions to McKean--Vlasov SDEs with coefficients of Nemytskii-type in the degenerate case.
We consider a broad class of degenerate nonlinear Fokker--Planck(--Kolmogorov) equations with coefficients of Nemytskii-type. This includes, in particular, the classical porous medium equation perturbed by a first-order term with initial datum in a subset of probability densities, which is dense with respect to the topology inherited from $L^1$, and, in the one-dimensional setting, the classical porous medium equation with initial datum in an arbitrary point $x_0\in\RR$.
For these kind of equations the existence of a Schwartz-distributional solution $u$ is well-known. We show that there exists a unique strong solution to the associated degenerate McKean--Vlasov SDE with time marginal law densities $u$. In particular, every weak solution to this equation with time marginal law densities $u$ can be written as a functional of the driving Brownian motion.
 Moreover, plugging any Brownian motion into this very functional yields a weak solution with time marginal law densities $u$.
\end{abstract}
\vspace{0.5em}
\noindent\textbf{Mathematics Subject Classification (2020):} 60H10,
60G17,
35C99.
\\
\sloppy \textbf{Keywords:} McKean--Vlasov stochastic differential equation, degenerate, pathwise uniqueness, Yamada--Watanabe theorem, nonlinear Fokker--Planck--Kolmogorov equation, porous medium equation.
\section{Introduction}\label{section.introduction}
We consider the following McKean--Vlasov stochastic differential equation (abbreviated by McKean--Vlasov SDE or MVSDE) in $\rd$, $d\in \LN$,  with coefficients of Nemytskii-type, which in our case is of the form
\begin{align}\label{MVSDE}
	\mathrm  dX(t) \notag
	=&\ E(X(t))b\left(\frac{\mathrm d\law{X(t)}}{\dx}(X(t))\right)\mathrm  dt  + \sqrt{2\frac{\beta\left(\frac{\mathrm d\law{X(t)}}{\dx}(X(t))\right)}{\frac{\mathrm d\law{X(t)}}{\dx}(X(t))}}\mathbbm{1}_{d\times d}\ \mathrm  d W(t), \notag \\
	X(0)=&\ X_0,\tag{MVSDE}
\end{align}
where $t\in [0,T]$, $T\in (0,\infty)$, $\mathbbm{1}_{d\times d}$ is the $d$-dimensional unit matrix, $(W(t))_{t\in[0,T]}$ is a standard $d$-dimensional  $(\FF_t)$-Brownian motion and $X_0$ an $\FF_0$-measurable function on some stochastic basis  $(\Omega,\FF,\PP;(\FF_t)_{t\in[0,T]})$, i.e.  a complete, filtered probability space, where $(\FF_t)_{t\in[0,T]}$ is a normal filtration, and $\law{X(t)}:= \PP\circ (X(t))\inv, t\in [0,T],$ denote the one-dimensional time marginal laws of the solution process $X$.
Here, we assume that the drift and diffusion coefficient are given through locally bounded Borel measurable functions
\begin{align*}
	E : \rd  \to \rd,\ b:\RR\to\RR,\ 
	\beta: \RR \to \RR,
\end{align*}
which satisfy conditions \ref{condition.a.general}, \ref{condition.b.general} below with the convention $\beta(0)\slash 0 := \beta'(0)$.
The McKean--Vlasov SDE \eqref{MVSDE} arises from the study of the following type of nonlinear Fokker--Planck(--Kolmogorov) equation (FPE)
\begin{align}
	\partial_t u(t,x) + \divv(E(x)b(u(t,x))u(t,x))-\Delta(\beta(u(t,x)))&=0,\ \ (t,x)\in (0,T)\times\rd, \notag\\
	\left.u\right|_{t=0} &= \nu \in M_b(\rd),\label{FPKE}\tag{FPE}
\end{align}
(here $M_b(\rd)$ denotes the set of bounded (signed) Borel measures on $\rd$).
In \cite{barbu2019nonlinear}, the authors consider also the case of a general non-diagonal diffusion matrix with explicit $x$-dependence. For simplicity we consider nonlinear Fokker--Planck--Kolmogorov equations of the above type.
As usual, \eqref{FPKE} is considered in the Schwartz-distributional sense. We will say that $u:(0,T)\to L^1(\rd)$, in short form $(u_t)_{t\in (0,T)}$, is a Schwartz-distributional solution to \eqref{FPKE} with $\left.u\right|_{t=0}=\nu \in M_b(\rd)$  if $u ,\beta(u) \in L^1_{loc}([0,T)\times\rd)$, 
$[0,T)\ni t\mapsto \int \varphi \left(\mathbbm{1}_{(0,T)}(t) u(t,x)\dx  + \mathbbm{1}_{\{0\}}(t)\mathrm d\nu\right)$ is continuous for all $\varphi \in C_b(\rd)$ and
\begin{align}\label{FPKE.test}
 0=\int_0^T\int_{\RR^d} &\left\langle\partial_t\varphi(s,x) + E(x)b(u(s,x)), \nabla \varphi(s,x)
\right\rangle_{\rd}u(s,x) + \Delta\varphi(s,x) \beta(u(s,x)) \dx\ds \nonumber \\
 &+ \int_\rd \varphi(0,x)\nu(\dx),
\end{align}
 for all $\varphi \in C_c^\infty([0,T)\times \rd)$. 
 If, additionally, $\nu \in \PPPP(\rd)$ and $u(t,\cdot) \in \PPPP(\rd)$ for all $t \in (0,T)$, then $u$ is simply called a \textit{probability solution} to \eqref{FPKE} (with $\left.u\right|_{t=0}=\nu$). Note that a probability solution $u$ can be uniquely extended to a continuous function $u:[0,T]\to \PPPP(\rd)$, where we consider the topology of narrow convergence on $\PPPP(\rd)$, see \cite[Lemma 2.3]{rehmeier2022nonlinearFlow}.

In applications, \eqref{FPKE} is used to describe particle transport in disordered media. A prominent particular example for this kind of equation is the classical porous medium equation, in which case $\beta(r)=|r|^{m-1}r, r\in \RR, m>1, E\equiv b\equiv 0$. It is most known for the usage as a model of diffusion of an ideal gas in a homogenous porous medium. For a deep account on the porous medium equation and equations of the general form \eqref{FPKE}, we refer to \cite{vazquez2007PME}, \cite{frank2005nlfpke} and the references therein, respectively.

In \cite{barbu2019nonlinear} (see also \cite{barbu2018Prob}), Barbu and R\"ockner introduced a general approach to solve \eqref{MVSDE} weakly in the probabilistic sense by first proving the existence of a Schwartz-distributional solution $u$ to \eqref{FPKE} and then, via a superposition principle procedure based on Trevisan's superposition principle for SDEs (\cite{trevisan_super}; see also the recent generalisation \cite{bogachev2021super} and a superposition principle in the nonlocal case \cite{roeckner2020levy}), constructing a probabilistically weak solution $(X,W)$ to \eqref{MVSDE}; this is done in such a way that the time marginal law densities are given by $u$, i.e. $\law{X(t)}=u_t(x)\dx$.  This procedure was applied under various conditions on the drift and diffusion coefficients, see \cite{barbu2020solutions, barbu2021evolution, barbu2021nonlinear}. For related uniqueness results we refer to \cite{barbu2019uniqueness, barbu2022uniqDeg}.

The above sketched approach constitutes a vital part to McKean's idea. In \cite{mckean1966class}, McKean envisioned a connection between certain classes of nonlinear PDEs, including the viscous Burgers' equation or the classical porous medium equations (in dimension one), and nonlinear Markov processes in such a way that the latter's transition probabilities solve the PDE.
Recently, Rehmeier and R\"ockner created a rigorous notion for a nonlinear Markov process in line with McKean's idea, see \cite{rehmeier2023nonlinear}. The authors proved the existence of nonlinear Markov processes for a large class of \eqref{MVSDE} including, in particular, those which are associated to the viscous Burgers' equation and the classical porous medium equation even in the multi-dimensional case. Furthermore, their results encompass the 2D vorticity Navier--Stokes equation and a large class of non-local Fokker--Planck equations.\\

The aim of this paper is to show that the probabilistically weak solutions for a wide class of \textit{degenerate} (i.e. $\beta(r)\slash r$ is allowed to vanish) McKean--Vlasov SDEs of type \eqref{MVSDE} provided by \cite{barbu2020solutions} are actually probabilistically strong solutions and pathwise unique among all weak solutions with time marginal law densities $u$, where $u$ denotes the probability solution constructed in \cite{barbu2020solutions}.
The class of McKean--Vlasov SDEs includes those associated to the classical porous medium equation perturbed by a nonlinear hyperbolic term of tensor type.
In order to be exact, we impose the following conditions on the coefficients of \eqref{MVSDE}.\\

Assume there exist $m\in (1,\infty)$ and $\zeta \in [0,1]$ with $\frac{2\zeta}{m}<1$ such that the following assumptions hold.
\begin{enumerate}[label={(H\arabic*)} , wide=0.5em,  leftmargin=*]
\item \label{condition.a.general} $\beta \in C^2(\RR)$; $\beta$ strictly increasing; $\beta(0)=0$; \\
	$\forall K>0$  $\exists a_K>0, C_K>0\ \forall r\in [0,K] : a_K r^{m-1} \leq \beta'(r)$ and  $\beta(r) \leq C_K r^{m}$.
\item \label{condition.b.general}
	\begin{enumerate}[label=(\alph*)]
		\item \label{condition.b.general:E} $E\in L^\infty(\rd;\rd), \divv E \in L^2_{loc}(\rd), (\divv E)^- \in L^\infty(\rd)$, and for all $R>0$ there exists $\iota_R \in L^1(B_R(0))$ such that for a.e. $x, y \in B_R(0)$
			\begin{align}\label{condition.E.monotonicity}
				\langle E(x)-E(y),x-y\rangle_{\rd} \leq \left(\iota_R(x) + \iota_R(y)\right)|x-y|_{\rd}^2.
			\end{align}
		\item \label{condition.b.general:b}
		\begin{enumerate}[label=(\roman*)]
			\item \label{condition.b.general:b.1} $b \in C_b(\RR)\cap C^1(\RR)$; $b\geq 0$; 
			\item\label{condition.b.general:b.2} $b\circ (G\circ\beta^\frac{1}{2})\inv \in\Lip_{loc}([0,\infty))$, where $G(r) := \int_0^r ((\beta\inv) (s^2))^{-\zeta}\ds, r\geq 0$.\vspace{0.5em}
			\end{enumerate}
			Additionally, one of the following conditions is assumed:
			\begin{enumerate}[leftmargin=1.4cm, label=(\roman*)]\setcounter{enumiii}{2}
				\item \label{condition.b.general:b.3} $\forall K>0\ \exists C_K >0\ \forall r\in [0,K]: \beta'(r)r \leq C_K \beta(r)$;
				
			\end{enumerate}
			\ \ or 
			\begin{enumerate}[leftmargin=1.4cm, label=(\roman*)]\setcounter{enumiii}{3}
				\item \label{condition.b.general:b.4} $E\in L^2(\rd;\rd)$, $\divv E \in L^2(\rd)+ L^\infty(\rd)$, $\divv E\geq 0$.
			\end{enumerate}
	\end{enumerate}
\end{enumerate}
\begin{remark}
	\begin{enumerate}[label=(\roman*)]
		\item Assume for a moment that $\beta$ is only assumed to be locally bounded. Then, the condition
		\begin{align*}
			\forall K>0\ \exists C_K>0\ \forall r\in [0,K]: \beta(r) \leq C_K r^m,
		\end{align*}
		is clearly equivalent to
		\begin{align*}
			\exists \varepsilon >0,C>0\ \forall r\in [0,\varepsilon]: \beta(r) \leq C r^m.
		\end{align*} 
		\item \ref{condition.a.general} implies that
			\begin{align*}
				\frac{\beta(r)}{r}\geq 0 \text{ with } \frac{\beta(0)}{0}=\beta'(0)= 0,
			\end{align*}
			and, therefore, that \eqref{MVSDE} is degenerate. Indeed, for some constant $C>0$,
			\begin{align*}
				\left|\frac{\beta(r)}{r}\right|\leq C|r|^{m-1} \to 0,\text{ as } r\downarrow 0.
			\end{align*}
		\item If $E \in H^1_{loc}(\rd;\rd)$, then clearly $\divv E \in L^2_{loc}(\rd)$ and, setting $\iota_R:= \M_{2R}|\nabla E|\ (\in L^2_{loc}(\rd))$, where $\M_{2R}$ is the local maximal function (see Definition \ref{appendix.definition.localmaximalfunction}), $E$ also fulfills \eqref{condition.E.monotonicity} by Lemma \ref{appendix.lemma.lipschitztypeestimate} and Lemma \ref{appendix.lemma.boundednesslocalmaximalfunction} from the Appendix.
		\item Note that, under condition \ref{condition.a.general},
			$G$, as introduced in \ref{condition.b.general} \ref{condition.b.general:b} \ref{condition.b.general:b.2}, is a real-valued function.
		\item \label{remark.conditons.global.5} Assume that \ref{condition.a.general} holds. Then, $G\inv \in C^1([0,\infty))$. Hence, additionally assuming \ref{condition.b.general}  \ref{condition.b.general:b} \ref{condition.b.general:b.1}, $b \circ (\beta^\frac{1}{2})\inv \in \Lip_{loc}(\RR)$ implies \ref{condition.b.general}  \ref{condition.b.general:b} \ref{condition.b.general:b.2}. Indeed, the first claim ist true: Since $G$ is bijective from $[0,\infty)$ to $[0,\infty)$, $G\inv$ exists. Furthermore, since $G'>0$ on $(0,\infty)$, $G\inv$ is continuously differentiable on $(0,\infty)$ with $(G\inv)'(r)= (\beta\inv((G\inv(r))^2	))^\zeta, r\in (0,\infty)$.
		\item Consider the case when $\beta(r) = |r|^{m-1}r, r\in \RR$. Assume that $b$ satisfies \ref{condition.b.general} \ref{condition.b.general:b} \ref{condition.b.general:b.1} and that $b(r)=r^l, r \in [0,\varepsilon]$ for some $\varepsilon>0$ and $l\geq 1$. Then, \ref{condition.b.general} \ref{condition.b.general:b} \ref{condition.b.general:b.2} is satisfied if and only if $l \geq \frac{m}{2}-\zeta$.
	\end{enumerate}
\end{remark}

In the following, $\PPPP_0(\RR^d)$ denotes the set of all probability densities with respect to the $n$-dimensional Lebesgue measure.
Also, we set
\begin{align*}
	\mathcal{Q}:=\{\text{$\Phi \in C^2(\rd)$ :  $\Phi \geq 1$; $\Phi^{-\alpha} \in L^1(\rd)$ for some $\alpha \in [2,\infty)$; $|\nabla\Phi|,\Delta \Phi \in L^\infty(\rd)$}\}
\end{align*}
Note that $\mathcal{Q}\neq \emptyset$, since $\Phi \in \mathcal{Q}$ with
$\Phi(x) := (1+|x|^2)^\gamma, x\in \rd,$ with $\gamma \in \left(0,\frac{1}{2}\right]$, see \cite[Proof of Lemma 3.3]{barbu2020solutions}.
Furthermore, we define for nonnegative Borel measurable $\Phi:\rd \to \RR$
$$L^1_\Phi(\rd):=\left\{f: \rd \to [-\infty,\infty] \ \ \BBBB(\rd)\text{-measurable} : ||f||_{L^1_\Phi}:=\int_\rd |f(x)| \Phi(x)\dx <\infty\right\}.$$

The first main result of this work says that under the conditions above, there exists a probabilistically strong solution to \eqref{MVSDE}, which is pathwise unique among weak solutions with time marginal law densities $u$ provided by Theorem \ref{theorem.FPKE.existence}. The corresponding precise statement in form of Theorem \ref{theorem.mainresult.1} is formulated in the terminology introduced on page \pageref{section.notation.solution}.
\begin{theorem}[main result 1]\label{theorem.mainresult.1}
	Let $d\neq 2$. Assume that \ref{condition.a.general} and \ref{condition.b.general} hold. Let $\nu \in \PPPP_0(\rd)\cap L^\infty(\rd)$ for which there exists $\Phi \in \mathcal{Q}$ such that $\Phi\in L^1_\Phi(\rd)$. Then, for every $T>0$, there exists a unique $P^{(u_t)}_\nu$-strong solution to \eqref{MVSDE} up to time $T$, where $u$ is the probability solution to \eqref{FPKE} with $\left.u\right|_{t=0}=\nu$ provided by Theorem \ref{theorem.FPKE.existence}.
\end{theorem}
Our second main result  says that, in the one-dimensional case, for every $x_0\in\RR$, there exists a strong solution to \eqref{MVSDE} with $X_0 := x_0$, which is pathwise unique among all weak solutions with time marginal law densities given by the Barenblatt solution to the classical porous medium equation.
\begin{theorem}[main result 2]\label{theorem.mainresult.2}
	Let $d=1$, $m>1$, $x_0\in \RR$. Furthermore, let $\beta(r)=|r|^{m-1}r, r\in \RR$, and $E\equiv b\equiv 0$. Let $u$ denote the Barenblatt solution to \eqref{FPKE} with $\left.u\right|_{t=0}=\delta_{x_0}$ as in \eqref{FPKE.PME.BarenblattPattleSolution} below. Then, for every $T>0$, there exists a unique $P^{(u_t)}_{\delta_{x_0}}$-strong solution to \eqref{MVSDE} up to time $T$.
\end{theorem}
In contrast to the well-known nondegenerate case (see \cite{jourdain1998propagation, bossy2019moderated, grube2021strong, grube2022strongTime}), there are only few results on the existence of strong solutions to McKean--Vlasov SDEs with coefficients of Nemytskii-type in the degenerate case.
In \cite{Benachour96}, the authors constructed a strong solution to \eqref{MVSDE} associated to the one-dimensional classical porous medium equation with initial datum $\nu \in \PPPP_0(\RR)\cap C_b(\RR)$ such that $\int_\RR \nu(x)|x|^n dx<\infty$, where $n> 2$, and $-\partial_{x}^2 (\nu^{m-1})\leq C$ in the sense of Schwartz-distributions. In the case $m>3$, the conditions on the initial data $\nu$ can be relaxed to $\nu \in \PPPP_0(\RR) \cap L^\infty(\RR)$ as was shown in \cite{grube2021strong} (see also \cite[Section 4.5.2]{grube2022thesis}).
Now, with Theorem \ref{theorem.mainresult.1}, we are able to treat the multidimensional case and, furthermore, even allow for first-order perturbations. 
To the best of the authors' knowledge, both of these aspects have not been considered in the literature so far.
Also, the existence of probabilistically strong solutions to McKean--Vlasov SDEs of the type \eqref{MVSDE} seems to not have been considered in the literature so far in the case when $X_0$ is distributed with a singular part with respect to Lebesgue measure. Therefore, to the best of the authors' knowledge, the result formulated in Theorem \ref{theorem.mainresult.2} is new as well.

The proofs of both our main results relies on the restricted Yamada--Watanabe theorem, Theorem \ref{theorem.restrictedYamadaWatanabe}, from \cite{grube2021strong} (see also \cite{grube2022thesis}).
Since the existence of weak solutions has already been proved in \cite{barbu2020solutions}, the task of this paper is to prove suitable pathwise uniqueness results. 
The most important ingredient for the pathwise uniqueness result related to Theorem \ref{theorem.mainresult.1} is Theorem \ref{theorem.FPKE.regularity.l1linfty}.
There we prove that the in \cite{barbu2020solutions} constructed Schwartz-distributional solution $u$ to \eqref{FPKE} with $\left.u\right|_{t=0}=\nu \in L^\infty(\rd)\cap L^1_\Phi(\rd)$, where $\nu \geq 0$ a.e. and $\Phi \in \mathcal{Q}$, satisfies the following estimate:
For all $T>0$ there exists $C_T\geq 0$ such that for all $t\in (0,T)$
	\begin{align}\label{FPKE.estimate.regularity.beta12}
		\int_\rd \Psi(u_t(x))\dx + \int_0^t \int_\rd |\nabla_x \beta^{\frac{1}{2}}(u(s,x))|^2\dx\ds\leq\int_\rd \Psi(\nu(x))\dx + C_T,
	\end{align}
	where $\Psi(r):= \int_0^r \ln(\beta(s))\ds, r\geq 0$. This result is strongly motivated by \cite{gess2023inventiones}.
	Fehrmann and Gess developed this kind of an estimate in order to prove, in particular, new existence and regularity theorems for analytically weak solutions to a large class of $L^1$ energy critical FPEs of the form \eqref{FPKE} on the $d$-dimensional torus with drift coefficients $(t,x,r) \mapsto E(t,x)b(r)$, where $E \in L^2$, $b(r)=\frac{\beta^\frac{1}{2}(r)}{r}$, and $\beta$ is taken from a large class of functions, which includes those of monomial type, namely $\beta(r)=|r|^{m-1}r$ for $m\geq 1$. 
	For \eqref{FPKE.estimate.regularity.beta12}, we use a similar idea of proof with differences owed to the transfer of their result to the mild solution framework in $L^1(\rd)$ (within the Crandall--Ligget theory), in which the Schwartz-distributional solution $u$ from above was constructed in \cite{barbu2020solutions}.
	
	The combination of the so obtained regularity of $\beta^\frac{1}{2}(u)$ and an estimate based on the mean-value theorem motived by \cite[(76)-(77)]{gess2023inventiones}, see e.g. \eqref{theorem.MVSDE.PU:trick1}, are key to our pathwise uniqueness result, Theorem \ref{theorem.MVSDE.PU}. The latter is a proper modification of the restricted pathwise uniqueness result in the proof of \cite[Theorem 1.1]{roeckner2010weakuniqueness}.
	The pathwise uniqueness result related to our second main result, Theorem \ref{theorem.mainresult.2}, is formulated in Theorem \ref{theorem.MVSDE.PME.PU}. Here, we already note that the Barenblatt solution $u$ does not satisfy \eqref{FPKE.estimate.regularity.beta12}, since $u^\frac{m}{2} \notin L^2((0,T);W^{1,2}(\rd))$. However, $u^\frac{m}{2} \in L^2((0,T);W^{s,2}(\rd))$, for all $s\in (0,1)$. This turns out to be enough to prove a suitable pathwise uniqueness result in dimension one. The proof is comparable to the one of Theorem \ref{theorem.MVSDE.PU}, but additionally, uses the technique of the one-dimensional result on the existence of strong solutions to rough (ordinary) stochastic differential equations \cite[Theorem 2.15 (i)]{champagnat2018} based on a Lipschitz type estimate for functions with fractional Sobolev regularity of order $\nicefrac{1}{2}$.
	\\

Within this work (as well as in \cite{grube2021strong,grube2022strongTime}), we consider \eqref{MVSDE} as a McKean--Vlasov SDE of the form
\begin{align*}
	dX(t) \notag
	=&\ \bm{b}(t,X(t),\law{X(t)})\mathrm dt  + \bm{\sigma}(t,X(t),\law{X(t)}) \mathrm dW(t),\ \ t\in [0,T], \notag \\
	X(0)=&\ X_0,
\end{align*}
with coefficients
\begin{align*}
	[0,T]\times\rd\times\PPPP(\rd)\ni(t,x,\nu) &\mapsto \bm{b}(t,x,\nu):= E(x)b(v_a(x)),\\
	[0,T]\times\rd\times\PPPP(\rd)\ni(t,x,\nu) &\mapsto \bm{\sigma}(t,x,\nu):=\sqrt{2\frac{\beta(v_a(x))}{v_a(x)}}\mathbbm{1}_{d\times d},
\end{align*}
where $\PPPP(\rd)$ is the set of all Borel probability measures on $\rd$ equipped with the topology of narrow convergence of probability measures, $v_a$ denotes the version of the density of the absolutely continuous part of $\nu$ both with respect to Lebesgue measure which is obtained by setting $v_a(x)= 0$ if $\lim_{R\to 0}\nicefrac{\nu(B_R(x))}{\lambda^d(B_R(0))}$ does not exist in $\RR$. By the Besicovitch derivation theorem, this makes the map $\rd\times\PPPP(\rd)\ni(x,\nu)\mapsto v_a(x)$ and therefore also
$\bm{b}$ and $\bm{\sigma}$ Borel measurable (for details see \cite[Section 4.2]{grube2022thesis}).
The dependence of $\bm{b}$ and $\bm{\sigma}$ on $\nu$ in terms of $v_a$ evaluated at a fixed point $x$ excludes the continuity of $\bm{b}$ and $\bm{\sigma}$ in their measure-component with respect to the topology of weak convergence of probability measures, Wasserstein distance or bounded variation norm.
These types of continuity assumptions are made in the major part of the literature (see, e.g. \cite{delarue2018mckean}).
\\

This paper is structured as follows.
	First, we introduce some necessary notation and recall the solution and uniqueness concepts for \eqref{MVSDE} from \cite{grube2021strong}.
	In Secton \ref{section.restrictedYamadaWatanabeTheorem}, we recall the restricted Yamada--Watanabe theorem from \cite{grube2021strong}, which is employed in the proof of the main results Theorem \ref{theorem.mainresult.1} and Theorem \ref{theorem.mainresult.2}.
	In Section 3, we prove the first main result, Theorem \ref{theorem.mainresult.1}. In Section 4, we prove the second main result, Theorem \ref{theorem.mainresult.2}.
	In Appendix \ref{section.appendix.maximalfunction}, we gather the  definition of and some facts on the (local) Hardy--Littlewood maximal function. In Appendix \ref{section.appendix.mildsolutionframework}, we recapitulate the mild solution framework according to Crandall and Ligget in a concise way.
	
\section*{Notation}
Within this paper we will use the following notation, which is essentially taken from \cite{grube2021strong}.
For a topological space $(\textbf{T},\tau)$, $\BBBB(\textbf{T})$ shall denote the Borel $\sigma$-algebra on $(\textbf{T},\tau)$. Throughout this article, $d,n,k$ denote natural numbers.
 
On $\RR^n$, we will always consider the usual $n$-dimensional Lebesgue measure $\lambda^n$ if not said any differently. If there is no risk for confusion, we will just say that some property for elements in $\RR^n$ holds \textit{almost everywhere} (or $\textit{a.e.}$) if and only if it holds $\lambda^n$-almost everywhere.
Furthermore, on $\RR^n$, $|\cdot|_{\RR^n}$ denotes the Euclidean norm. If there is no risk for confusion, we will just write $|\cdot|=|\cdot|_{\RR^n}$. By $B_R(x)$ we will denote the usual open ball with center $x\in \RR^n$ and radius $R>0$. Furthermore, we set $s\wedge t := \mathrm{min}(t,s),\ s\vee t := \mathrm{max}(t,s)$ for all $s,t \in \RR$.

Let $(S,\SSS,\eta)$ be a measure space. For $1\leq p \leq \infty$, $(L^p(S;E), \norm[{L^p(S;E)}]{\cdot})$ symbolises the usual Bochner space of strongly measurable $E$-valued functions $f$ on $S$ for which $\norm[E]{f}^p$ is integrable. 
If $S=\RR^n$ and $E=\RR$, we just write $L^p(\RR^n;\RR)=L^p(\RR^n)$. The set of strongly measurable functions on $\RR^n$ with values in $E$ which are locally $p$-integrable in norm on $E$ will be denoted by $L^p_{loc}(\RR^n;E)$.
Moreover, $(W^{k,p}(\RR^n), \norm[{W^{1,p}(S;E)}]{\cdot})$ denotes the usual Sobolev space, containing all $L^p(\RR^n)$-functions, whose Schwartz-distributional derivatives up to $k$-th order can be represented by elements in $L^p(\RR^n)$. 
Accordingly, $E$-valued first-order Sobolev functions on $\RR^n$ will be denoted by $W^{k,p}(\RR^n;E)$. In the special case when $p=2$, we also write $W^{k,2}=H^k$.
For $\eta\in (0,\infty)$ of the form $\eta=s+k$, where $s\in (0,1)$ and $k\in \LN_0$, $\dot{W}^{\eta,p}(\RR^n)$ denotes the homogeneous Sobolev--Slobodeckij space, which consists of all functions $f\in L^1_{loc}(\RR^n)$ such that for all $\alpha \in \LN^k$ with $|\alpha|=k$
\begin{align*}
	[D^\alpha f]_{\dot{W}^{s,p}(\RR^n)} :=\left( \int_\rd\int_\rd \frac{|D^\alpha f(x)-D^\alpha f(y)|^p}{|x-y|^{sp+d}}\dx\dy\right)^\frac{1}{p} <\infty,
\end{align*}
with norm $||f||_{\dot{W}^{\eta,p}}:=\sup_{|\alpha|=k}[D^\alpha f]_{\dot{W}^{s,p}(\RR^n)}$.
Here, $D^\alpha$ denotes the usual (Schwartz-distributional) derivative of order $k$, where for each $i\in \{1,...,d\}$, the derivative in direction $x_i$ is taken $\alpha_i$-times.
The nonhomogeneous Sobolev--Slobodeckij space is defined as $W^{\eta,p}(\RR^n):= L^p(\rd) \cap \dot{W}^{\eta,p}(\RR^n)$ with norm $||f||_{W^{\eta,p}(\rd)}:=||f||_{L^p(\rd)}+||f||_{\dot{W}^{\eta,p}(\RR^n)}$. Both the homogenous and nonhomogenous Sobolev--Slobodeckij space are Banach spaces.
Recall that $W^{s,2}(\RR^n)= \{f \in L^2(\RR^n) : \mathcal{F}\inv (|\xi|^s \mathcal{F}(f))\in L^2(\RR^n)\}$. Here, $\mathcal{F}$ denotes the Fourier transform on the tempered Schwartz-distributions. For $f \in W^{s,2}(\RR^n)$ we set $\nabla^sf := \mathcal{F}\inv (|\xi|^s \mathcal{F}(f))$, which is also known as the Riesz derivative. 
Furthermore, for every real-valued function, we set $f^- := -\min (f,0)$ and $f^+ := \max(f,0)$.

Let $(M,d)$ be a metric space. Then $\PPPP(M)$ denotes the set of all Borel probability measures on $(M,d)$. We will consider $\PPPP(M)$ as a topological space with respect to the topology of narrow convergence of probability measures.
By $\PPPP_0(\RR^n)$ we will denote the set of all probability densities with respect to Lebesgue measure, i.e.
\begin{align*}
	\PPPP_0(\RR^n) = \left\{ \rho \in L^1(\RR^n)\ |\ \rho\geq 0 \text{ a.e.}, \int_{\RR^n} \rho(x)\dx =1\right\}.
\end{align*}

By $C_c^\infty(\RR^n)$ we denote the set of all infinitely differentiable functions with compact support. 
Let $(E,\norm[E]{\cdot})$ be a Banach space.
The set of continuous functions on the interval $[0,T]$ with values in $E$ is denoted by $C([0,T];E)$ and is considered with respect to the usual supremum's norm.
Furthermore, we define $$C([0,T];E)_0 := \{w \in C([0,T];E): w(0)=0\}.$$
For $t\in [0,T]$, $\pi_t: C([0,T];E) \to E$ denotes the canonical evaluation map at time $t$, i.e. $\pi_t(w):=w(t), w\in C([0,T];E)$. Further, we set $\BBBB_t(C([0,T];E)):= \sigma(\pi_s : s\in [0,t])$ and, correspondingly, $\BBBB_t(C([0,T];E)_0):= \sigma(\pi_s: s\in [0,t])\cap C([0,T];E)_0$.
Moreover, $\PP_W$ denotes the Wiener measure on $(C([0,T];\rd)_0,\BBBB(C([0,T];\rd)_0))$.

Throughout this article, $C\in (0,\infty)$ will denote a generic constant, which may change from line to line and which may depend on the spacial dimension $d$ and the final (finite) time $T$.
The derivative $\partial_z$ denotes the derivative with respect to a scalar $z$-coordinate. Furthermore, $\divv=\divv_x$, $\nabla = \nabla_x$, $\Delta = \Delta_x$ and $D=D_x$ denote the divergence, gradient, Laplacian and Jacobian with respect to the spacial $x$-coordinate. All the derivatives are supposed to be understood in the sense of Schwartz-distributions.
\section*{$P^{(u_t)}_\nu$-solutions to \eqref{MVSDE}}
\label{section.notation.solution} 
Let us briefly recall the solution concepts for \eqref{MVSDE} from \cite{grube2021strong} (see also \cite{grube2022thesis}).

Let $\nu \in \PPPP(\rd)$ and suppose $u$ is a probability solution to \eqref{FPKE} with $\left.u\right|_{t=0}=\nu$. We set
\begin{align*}
	P^{(u_t)}_\nu:=\{Q \in \PPPP(C([0,T];\rd)): Q \circ \pi_0\inv = \nu, Q \circ \pi_t\inv = u_t(x)\dx\  \forall t\in (0,T]\}.
\end{align*}

A \textbf{$P^{(u_t)}_\nu$-weak solution} $(X,W,(\Omega,\FF,\PP;(\FF_t)_{t\in[0,T]}))$ is a (probabilistically) weak solution $(X,W, (\Omega,\FF,\PP;(\FF_t)_{t\in[0,T]}))$ to \eqref{MVSDE} in the usual sense such that $\law{X(0)}=\nu$ and $\law{X(t)} = u_t(x)\dx$, for all $t\in (0,T]$. For the convenience of the reader, we will just write $(X,W)=(X,W, (\Omega,\FF,\PP;(\FF_t)_{t\in[0,T]}))$ in cases in which we do not need to refer explicitly to the underlying stochastic basis $(\Omega,\FF,\PP;(\FF_t)_{t\in[0,T]})$.

We will say that \eqref{MVSDE} has a \textbf{$P^{(u_t)}_\nu$-strong solution} if there exists a function $F: \rd \times C([0,T];\rd)_0\to C([0,T];\rd)$, which is $\overline{\BBBB(\rd)\otimes \BBBB(C([0,T];\rd)_0)}^{\nu\otimes \PP_W}\slash \BBBB(C([0,T];\rd))$-measurable, such that, for $\nu$-a.e. $x\in \rd$, $F(x,\cdot)$ is $\overline{\BBBB_t(C([0,T];\rd)_0)}^{\PP_W}\slash \BBBB_t(C([0,T];\rd))$-measurable for all $t\in [0,T]$ and, whenever $X_0$ is an $\FF_0$-measurable function with $\law{X_0} = \nu$ and $W$ is a standard $d$-dimensional $(\FF_t)$-Brownian motion on some stochastic basis $(\Omega,\FF,\PP;(\FF_t)_{t\in [0,T]})$, $(F(X_0,W),W, (\Omega,\FF,\PP;(\FF_t)_{t\in[0,T]}))$ is a $P^{(u_t)}_\nu$-weak solution to \eqref{MVSDE}.
Here, $\overline{\BBBB(\rd)\otimes \BBBB(C([0,T];\rd)_0)}^{\nu\otimes \PP_W}$ denotes the completion of $\BBBB(\rd)\otimes \BBBB(C([0,T];\rd)_0)$ with respect to the measure $\nu\otimes \PP_W$, and $\overline{\BBBB_t(C([0,T];\rd)_0)}^{\PP_W}$ denotes the completion of $\BBBB_t(C([0,T];\rd)_0)$ with respect to $\PP_W$ on $(C([0,T];\rd)_0,\BBBB(C([0,T];\rd)_0))$.

Moreover, \textbf{$P^{(u_t)}_\nu$-pathwise uniqueness} holds for \eqref{MVSDE} if for every two $P^{(u_t)}_\nu$-weak solutions $(X,W, (\Omega,\FF,\PP;(\FF_t)_{t\in[0,T]}))$, $(Y,W, (\Omega,\FF,\PP;(\FF_t)_{t\in[0,T]}))$ (with respect to the same Brownian motion on the same stochastic basis) with $X(0)=Y(0)$ $\PP$-a.s., one has ${\sup_{t\in [0,T]}|X(t) - Y(t)|}=0$ $\PP$-a.s.

We say that there exists a \textbf{unique $P^{(u_t)}_\nu$-strong solution} to \eqref{MVSDE} if there exists a $P^{(u_t)}_\nu$-strong solution to \eqref{MVSDE} with functional $F$ as above, and every $P^{(u_t)}_\nu$-weak solution $(X,W)$ is of the form $X = F(X(0),W)$ almost surely with respect to the underlying probability measure.
\section{The restricted Yamada--Watanabe theorem}\label{section.restrictedYamadaWatanabeTheorem}
We recall the restricted Yamada--Watanabe theorem from \cite{grube2021strong} (see also \cite{grube2022thesis}), which is a modification of the original Yamada--Watanabe theorem. In particular, it has the advantage that one can conclude the existence of a strong solution to \eqref{MVSDE} under a relaxed pathwise uniqueness condition compared to the original Yamada--Watanabe theorem.

The following theorem can essentially be found in \cite[Theorem 3.3]{grube2021strong}. 
\begin{theorem}
\label{theorem.restrictedYamadaWatanabe}
Let $\nu \in \PPPP(\rd)$ and suppose $u$ is a probability solution to \eqref{FPKE} such that $\left.u\right|_{t=0}=\nu$.
The following statements regarding \eqref{MVSDE} are equivalent. 
\begin{enumerate}[label=(\roman*)]
	\item There exists a $P^{(u_t)}_\nu$-weak solution and $P^{(u_t)}_\nu$-pathwise uniqueness holds.
	\item There exists a unique $P^{(u_t)}_\nu$-strong solution to \eqref{MVSDE}.
\end{enumerate}
\end{theorem}
\section{Proof of Theorem \ref{theorem.mainresult.1}}\label{section.theorem.mainresult.1.proof}

In this section, we prove the first main result, Theorem \ref{theorem.mainresult.1}. This section is split into three subsections.
In Section \ref{section.theorem.FPKE.existence}, we will recall the results on the existence of solutions to \eqref{FPKE} from \cite{barbu2020solutions} and, moreover, prove a regularity result for $\beta^\frac{1}{2}(u)$, where $u$ is the solution to \eqref{FPKE}, which was constructed in \cite{barbu2020solutions}.
Here, the presentation of the techniques and results of \cite{barbu2020solutions} is similar to the one in \cite[Appendix D.2]{grube2022thesis}.
In Sections \ref{subsection.theorem.mainresult.1.proof:ingredient1} and \ref{subsection.theorem.mainresult.1.proof:ingredient2}, we then show that the conditions (i) and (ii) of the restricted Yamada--Watanabe, see Theorem \ref{theorem.restrictedYamadaWatanabe}, are fulfilled, respectively.

\begin{proof}[Proof of Theorem \ref{theorem.mainresult.1}]
	The assertion follows directly from Theorem \ref{theorem.MVSDE.existence} and Theorem \ref{theorem.MVSDE.PU} via Theorem \ref{theorem.restrictedYamadaWatanabe}.
\end{proof}

\subsection{Existence of a regular Schwartz-distributional solution to \eqref{FPKE}}\label{section.theorem.FPKE.existence}

In \cite{barbu2020solutions}, Barbu and R\"ockner proved the existence of a Schwartz-distributional solution to \eqref{FPKE} with initial datum $\nu \in L^1(\rd)\cap L^\infty(\rd)$ under even more general conditions on $E,b,\beta$ than \ref{condition.a.general}, \ref{condition.b.general}, namely
\begin{enumerate}[label=(C\arabic*) , wide=0.5em,  leftmargin=*]
	\item \label{condition.a.general.l1linfty} $\beta \in C^2(\RR)$, $\beta$ is monotonically nondecreasing, $\beta(0)=0$, 
	\item \label{condition.b.general.l1linfty} 
	$E\in L^\infty(\rd;\rd), \divv E \in L^2_{loc}(\rd), (\divv E)^- \in L^\infty(\rd)$,\\
	$b \in C_b(\RR)\cap C^1(\RR)$, $b\geq 0; b = \text{const.}$ if $\beta$ is not strictly increasing.
\end{enumerate}
Therefore, they considered \eqref{FPKE} (up to time $T=\infty$) as an evolution equation in $L^1(\rd)$ of the form
\begin{align}\label{FPKE.Op}\tag{FPE.Op}
	\partial_t u(t) + A(u(t))  &= 0, \ \ t \in [0,\infty),\\
			\left.u\right|_{t=0} &= \nu,\notag
\end{align}
where $(A,D(A))$ is a nonlinear operator acting on $L^1(\rd)$, which coincides with the operator $(A_0, D(A_0))$ defined as
\begin{align*}
	A_0 (u) &:= \divv(Eb(u)u)-\Delta\beta(u), u\in D(A_0),\\
	D(A_0)&:= \{ u \in L^1(\rd): \beta(u) \in L^1_{loc}(\rd), -\Delta \beta(u) + \divv(Eb(u)u) \in L^1(\rd)\},
\end{align*}
on a domain $D(A)\subset D(A_0)$ to be specified later on.
The following lemma is a special case of \cite[Lemma 3.1]{barbu2020solutions}.
\begin{lemma}[\text{cf. \cite[Lemma 3.1]{barbu2020solutions}}]\label{BR21.lemma3.1}
	Assume that \ref{condition.a.general.l1linfty} and \ref{condition.b.general.l1linfty} hold. We have
	\begin{align*}
		\mathrm{Range}(I+\lambda A_0 )=L^1(\rd) \ \ \ \forall \lambda >0,
	\end{align*}
	and there exists an operator $J_\lambda: L^1(\rd) \to L^1(\rd)$ such that
	\begin{align}\label{BR21.lemma3.1:1}
	J_\lambda (v) \in (I + \lambda A_0)\inv (v)\ \ \ &\forall v \in L^1(\rd),\notag\\
		\norm[L^1]{J_\lambda (v) - J_\lambda (\bar{v})	} \leq  \norm[L^1]{v- \bar{v}	}\ \ \ &\forall v, \bar{v} \in L^1(\rd), \lambda >0,\\
	\label{lemma.FPKE.A_epsilon:1}
	J_{\lambda_2}(v) = J_{\lambda_1}\left(\frac{\lambda_1}{\lambda_2}v+\left(1-\frac{\lambda_1}{\lambda_2}\right)J_{\lambda_2}(v)  \right) \ \ \ &\forall 0<\lambda_1, \lambda_2 < \infty.
	\end{align}
	Moreover,
	\begin{align}\label{BR21.lemma3.1:4}
	\norm[L^\infty]{J_\lambda (v)} \leq \left(1+ \norm[L^\infty]{(\divv E)^- + |E|}^\frac{1}{2}\right)\norm[L^\infty]{v} \ \ \ \forall v \in L^1(\rd)\cap L^\infty(\rd),\notag\\
	0<\lambda < \lambda_0 :=\left(\norm[L^\infty]{(\divv E)^- + |E|}+ \left(\norm[L^\infty]{(\divv E)^- + |E|}^\frac{1}{2} \right) \norm[L^\infty]{b}\right)\inv
	\end{align}
	and
		\begin{align}\label{BR21.lemma3.1:5}
		J_\lambda (\PPPP_0(\rd)) \subset \PPPP_0(\rd)\ \ \ \forall\lambda >0.
	\end{align}
\end{lemma}
With the results of Lemma \ref{BR21.lemma3.1}, Barbu and R\"ockner defined $D(A)$ as follows:
\begin{align}\label{FPKE.Op.A}
A (v) &:= A_0 (v), v\in D(A),\\
D(A) &:= \{J_\lambda(v): v\in L^1(\rd)\}, \lambda >0.\notag
\end{align}
Note that $D(A)$ is independent of the parameter $\lambda>0$, as can be seen from \eqref{lemma.FPKE.A_epsilon:1}, and is therefore well-defined.
From Lemma \ref{BR21.lemma3.1} the following corollary is immediately deduced.
For the terminology, we refer to Appendix \ref{section.appendix.mildsolutionframework}.
\begin{corollary}[{\cite[Lemma 3.2]{barbu2020solutions}}]Assume that \ref{condition.a.general.l1linfty} and \ref{condition.b.general.l1linfty} hold. Then,
	\begin{enumerate}[label=(\roman*)]
		\item $J_\lambda$ coincides with the inverse $(I+\lambda A)\inv$ of $(I+\lambda A)$.
		\item The operator $A$ is $m$-accretive in $L^1(\rd)$.
	\end{enumerate}
	Moreover, \eqref{BR21.lemma3.1:4} and \eqref{BR21.lemma3.1:5}  hold with $(I+\lambda A)\inv$ instead of $J_\lambda$. 
\end{corollary}
By the Crandall--Ligget theorem (cf. Theorem \ref{appendix.theorem.crandallLigget}), the following theorem is then proved.
It is a special case of \cite[Theorem 2.2]{barbu2020solutions}, which is one of the main results of \cite{barbu2020solutions}.
\begin{theorem}[c.f. {\cite[Theorem 2.2]{barbu2020solutions}}]\label{theorem.FPKE.existence}
	Let $d\neq 2$. Assume that \ref{condition.a.general.l1linfty} and \ref{condition.b.general.l1linfty} hold. Then $\overline{D(A)}^{L^1(\rd)}= L^1(\rd)$ and for every $\nu \in {L^1(\rd)}$ there exists a unique mild solution $u$ to \eqref{FPKE.Op} with $\left.u\right|_{t=0}=\nu$.
	Furthermore, for every $\nu\in {L^1(\rd)} \cap L^\infty(\rd)$
	\begin{align}\label{theorem.FPKE.existence:Linfty}
		\norm[L^\infty(\rd)]{u(t)} \leq \exp\left(\norm[L^\infty(\rd)]{(\divv E)^-+|E|}^\frac{1}{2}t\right)\norm[L^\infty(\rd)]{v} \ \forall t\in [0,\infty).
	\end{align}
	If $\nu \in \PPPP_0(\rd)$, then
	\begin{align}\label{theorem.FPKE.existence:P0}
		u(t) \in \PPPP_0(\rd)\ \forall t\in [0,\infty).
	\end{align}
	Moreover,
	$t\mapsto S(t)\nu=u(t,\cdot)$, where $u$ denotes the mild solution starting in $\nu$, is a strongly continuous semigroup of nonlinear contractions from $L^1(\rd)$ to $L^1(\rd)$, that is $S(t+s)\nu = S(t)S(s)\nu$ for all $0<s<t$, and 
	\begin{align*}
		||S(t)\nu - S(t)\bar{\nu}||_{L^1} \leq ||\nu - \bar{\nu}||_{L^1} \ \forall t>0,\ \nu, \bar{\nu} \in L^1(\rd).
	\end{align*}
	If $\nu \in L^1(\rd) \cap L^\infty(\rd)$, then $u$ is a Schwartz-distributional solution to \eqref{FPKE} up to $T=\infty$ such that $\left.u\right|_{t=0}=\nu$.
\end{theorem}
In view of the proof of Theorem \ref{theorem.FPKE.regularity.l1linfty} below, we need to provide more details on how Lemma \ref{BR21.lemma3.1} is proved.

\noindent The proof of Lemma \ref{BR21.lemma3.1} is extensive and is based on the analysis of the operators $(I+\lambda A_\varepsilon)\inv,\lambda>0,\varepsilon>0,$ where $(A_\varepsilon,D(A_\varepsilon))$ is defined as
\begin{align*}
	A_\varepsilon v &:= -\Delta\tilde{\beta}_\varepsilon(v) + \varepsilon \tilde{\beta}_\varepsilon(v) +\divv(E_\varepsilon b_\varepsilon(v)v), v \in D(A_\varepsilon),\\
	D(A_\varepsilon)&:= \{ v \in L^1(\rd):  -\Delta \tilde{\beta}_\varepsilon(v) +\varepsilon \tilde{\beta}_\varepsilon(v) + \divv(E_\varepsilon b_\varepsilon(v)v) \in L^1(\rd)\},
\end{align*}
where $\tilde{\beta}_\varepsilon(r):=\beta_\varepsilon(r)+\varepsilon r, r\in \RR$, and for $\varepsilon>0, r\in \RR$,
\begin{align*}
\beta_\varepsilon(r)&:=\frac{1}{\varepsilon}(r-g_\varepsilon(r))=\beta(g_\varepsilon(r))
\text{, where } g_\varepsilon(r):=(I+\varepsilon\beta)\inv(r),\\
E_\varepsilon&:=   \left\{
\begin{array}{ll}
     E, & \text{if $|E|\in L^2(\rd)$ and $\divv (E) \in L^2(\rd) + L^\infty(\rd)$,}\\
     \eta_\varepsilon E, & \text{else.}
\end{array}\right. \\
b_\varepsilon(r)&:=   \left\{
\begin{array}{ll}
     b, & \text{if $b$ is constant,}\\
     \frac{(b\ast\rho_\varepsilon)(r)}{1+\varepsilon|r|}, & \text{else.}
\end{array} 
\right. 
\end{align*}
Here $\rho_\varepsilon(r):= \varepsilon\inv \rho(\varepsilon\inv r)$, $\rho \in C_c^\infty(\RR)$, $\rho \geq 0$, is a standard mollifier and by $I$ we denote the identity map on $\RR$. Furthermore, $\eta_\varepsilon \in C_c^1(\rd)$, $0\leq \eta_\varepsilon\leq 1, |\nabla \eta_\varepsilon| \leq 1$ and $\eta_\varepsilon(x) =1$ if $|x|\leq \varepsilon\inv$.
Then we have $E_\varepsilon\in L^2(\rd;\rd)\cap L^\infty(\rd;\rd), |E_\varepsilon|\leq |E|, \lim_{\varepsilon\to 0} E_\varepsilon(x) = E(x)$ and 
\begin{align}\label{FPKE.approx.divEeps-.estimate}
	{(\divv E_\varepsilon)^- \leq (\divv E)^-+|E|\mathbbm{1}_{\{|\cdot|\geq\varepsilon\inv\}}}.
\end{align}
Consider the following equation in the Schwartz-distributional sense
\begin{align}\label{A_epsilon.ellipticEq}
	u_\varepsilon + \lambda A_\varepsilon (u_\varepsilon) = f,
\end{align}
for $f \in L^1(\rd)$ and $\lambda >0$.
We have the following lemma on the solvability of \eqref{A_epsilon.ellipticEq} in the case $f\in L^1(\rd)\cap L^2(\rd)$ together with fundamental estimates for a solution $u_\varepsilon$, which is extracted from \cite[Proof of Lemma 3.1]{barbu2020solutions}.
\begin{lemma}\label{lemma.FPKE.A_epsilon}
	Let $f\in L^1(\rd)\cap L^2(\rd)$ and $\lambda >0$. Then \eqref{A_epsilon.ellipticEq} has a solution $u_\varepsilon=u_\varepsilon(\lambda,f) \in L^1(\rd)\cap H^1(\rd)$ in the strong sense in $L^2(\rd)$.
	If $f\geq 0$ a.e. then $u_\varepsilon \geq 0$ a.e. 
	For $\bar{f} \in L^1(\rd)\cap L^2(\rd)$ we have
	\begin{align}\label{lemma.FPKE.A_epsilon:1}
		\norm[L^1(\rd)]{u_\varepsilon(\lambda,f)-u_\varepsilon(\lambda,\bar{f})} \leq \norm[L^1(\rd)]{f-\bar{f}},
	\end{align}
	and $u_\varepsilon(\lambda,0)=0$.
	Furthermore,
	\begin{align}\label{lemma.FPKE.A_epsilon:convergenceJlambda}
		u_\varepsilon(\lambda) \to J_\lambda f \text{ in } L^1(\rd), \text{ as } \varepsilon\to 0.
	\end{align}

	Let $f\in L^1(\rd)\cap L^\infty(\rd)$.
	Then for all $\lambda \in (0,\lambda_0)$, where $\lambda_0>0$ is defined as in \eqref{BR21.lemma3.1:4}, $u_\varepsilon(\lambda,f) \in L^\infty(\rd)$ such that
	\begin{align*}
	\norm[L^\infty(\rd)]{u_\varepsilon(\lambda,f)} \leq C_E\norm[L^\infty(\rd)]{f},
	\end{align*}
	where 
	\begin{align*}
		C_E:=\left(1+ \norm[L^\infty(\rd)]{(\divv E)^- + |E|}^\frac{1}{2}\right).
	\end{align*}
\end{lemma}
The result and proof of the following lemma can be found in \cite[Proof of Lemma 3.3]{barbu2020solutions}, see also \cite[Lemma 3.2]{barbu2021evolution}.
\begin{lemma}\label{lemma.BR21.Lemma3.3.proof} Let $f\in L^1_{\Phi}(\rd)\cap L^\infty(\rd)$ for some $\Phi \in \mathcal{Q}$. Let $\varepsilon\in (0,1), \lambda \in (0,\lambda_0)$, where $\lambda_0>0$ is as in \eqref{BR21.lemma3.1:4}. Let $u_\varepsilon:= u_\varepsilon(\lambda, f)$ denote the solution to \eqref{A_epsilon.ellipticEq} provided by Lemma \ref{lemma.FPKE.A_epsilon}. Then, we have
	\begin{align}
		||u_\varepsilon||_{L^1_\Phi} 
		&\leq ||f||_{L^1_\Phi} + \lambda\int |\tilde{\beta}_{\varepsilon}(u_\varepsilon)| \Delta \Phi + |b_\varepsilon(u_\varepsilon)u_\varepsilon| |\langle E_\varepsilon,\nabla \Phi\rangle_{\rd}|\dx.
	\end{align}
\end{lemma}
\begin{remark}\label{lemma.BR21.Lemma3.3.proof:remark}
	It is easy to see that Lemma \ref{lemma.BR21.Lemma3.3.proof} and \eqref{lemma.FPKE.A_epsilon:convergenceJlambda} imply the inequality
	\begin{align*}
		||J_\lambda(f)||_{L^1_\Phi} 
	&\leq ||f||_{L^1_\Phi} + \lambda\int |\beta(J_\lambda(f))| \Delta \Phi + |b(J_\lambda(f))J_\lambda(f)| |\langle E,\nabla \Phi\rangle_{\rd}|\dx.
	\end{align*}
	\end{remark}
The following theorem provides conditions under which it is guaranteed that $\beta^{\frac{1}{2}}(u) \in L^2([0,T];H^1(\rd))$. 
The technique and the estimate \eqref{theorem.FPKE.regularity.l1linfty:regularity.beta12} we obtain, are similar to those in \cite{gess2023inventiones}. As pointed out in Section \ref{section.introduction}, Fehrmann and Gess established an a priori estimate similar to \eqref{theorem.FPKE.regularity.l1linfty:regularity.beta12} on the $d$-dimensional torus $\mathbb{T}^d$ in order to prove the existence and regularity of analytically weak solutions to \eqref{FPKE} including those with drift coefficient $[0,T]\times\rd\times\RR\ni (t,x,r) \mapsto E(t,x)\frac{\beta^\frac{1}{2}(r)}{r}$, $E^i \in L^2([0,T]\times\mathbb{T}^d), i \in \{1,...,d\}$, and a large class of diffusivity functions $\beta$, including those of monomial type, i.e. $\beta(r)=|r|^{m-1}r$ for every $m\geq 1$. Here, we transfer their technique to \eqref{FPKE.Op} within the mild solution framework in $L^1(\rd)$.
\begin{theorem}
\label{theorem.FPKE.regularity.l1linfty}
	Assume that \ref{condition.a.general.l1linfty} and \ref{condition.b.general.l1linfty} hold. 
	Assume that 
	\begin{align}\label{theorem.FPKE.regularity.l1linfty:cond.E}
		E \in L^2(\rd;\rd),\ \divv E \in L^2(\rd)+ L^\infty(\rd),\ \divv E\geq 0;
	\end{align}
	or that
	\begin{align}\label{theorem.FPKE.regularity.l1linfty:cond.local}
		\forall K>0\ \exists C_K >0\ \forall r\in [0,K]: \beta'(r)r \leq C_K \beta(r).
	\end{align}
	Let $\nu \in L^1_\Phi(\rd)\cap L^\infty(\rd)$ for some $\Phi \in \mathcal{Q}$ with $\nu \geq 0$ a.e. and let $u$ denote the Schwartz-distributional solution to \eqref{FPKE} with $\left.u\right|_{t=0} = \nu$ provided by Theorem \ref{theorem.FPKE.existence}. Then, for all $T>0$ there exists $C_T>0$ such that for all $t\in (0,T)$
	\begin{align}\label{theorem.FPKE.regularity.l1linfty:regularity.beta12}
		\int_\rd \Psi(u_t(x))\dx + \int_0^t \int_\rd |\nabla \beta^{\frac{1}{2}}(u_s(x))|^2\dx\ds\leq\int_\rd \Psi(\nu(x))\dx + C_T,
	\end{align}
	where $\Psi(r):= \int_0^r \ln(\beta(s))ds, r\geq 0$. If \eqref{theorem.FPKE.regularity.l1linfty:cond.E} is assumed, then \eqref{theorem.FPKE.regularity.l1linfty:regularity.beta12} is true for $C_T=0$.
\end{theorem}
\begin{proof}
\newcommand{\logn}{l^\eta_{M,n}}
\newcommand{\lognbeta}{l_{M,n}(\cdot,\tilde{\beta}_{\varepsilon}(u_\varepsilon))}
\newcommand{\tildebeta}{\tilde{\beta}_\varepsilon(u_\varepsilon)}
\newcommand{\tildebetaderivative}{\tilde{\beta}'_\varepsilon(u_\varepsilon)}
\newcommand{\uepsiplus}{u_\varepsilon}
\newcommand{\uepsi}{f}

	By Definition \ref{appendix.crandallLigget.mildSolution}, we have for each $T>0$
	\begin{align*}
		u_{}(t)=\lim_{h\to 0} u_{h}(t) \text{ in $L^1(\rd)$, uniformly on $[0,T]$,}
	\end{align*}
	where $u_{h}:[0,T]\to L^1(\rd)$ is a step function, defined by the finite difference scheme
	\begin{align}\label{theorem.FPKE.regularity.l1linfty:A.finite.difference.scheme}
		u_{h}(t)=u_{h}^{i+1}\ \forall t\in (ih,(i+1)h], i=0,1,...,N-1, u_{h}(0)=\nu,\\
		\text{where } u_{h}^{i+1}+hA(u_{h}^{i+1}) = u_{h}^i,\ Nh=T,\notag\\
		\text{and } u_{h}^0=\nu.\notag
	\end{align}
	
	We prove the assertion of this lemma according to the following three steps.
	\begin{enumerate}
		\item We show that for every $f\in L^1_\Phi(\rd) \cap L^\infty(\rd)$ with $f\geq 0$ a.e. and $\lambda \in (0,\lambda_0)$,	the corresponding solution $u_\varepsilon=u_\varepsilon(\lambda,f)$ to  
			\begin{align}\label{theorem.FPKE.regularity.l1linfty:A_epsilon.eq}
				u_{\varepsilon}+\lambda A_\varepsilon(u_{\varepsilon}) = f,
			\end{align}
			provided by Lemma \ref{lemma.FPKE.A_epsilon},
			satisfies for $\Psi_\varepsilon(r):= \int_0^r \ln(\tilde{\beta}_\varepsilon(r))dr, r\geq 0$,
			\begin{align}\label{theorem.FPKE.regularity.l1linfty:regularity.beta12:1.1}
				\int_\rd \Psi_\varepsilon(u_\varepsilon)\dx + \lambda\int_\rd |\nabla \tilde{\beta}_\varepsilon^{\frac{1}{2}}(u_\varepsilon)|^2\dx\leq\int_\rd \Psi_\varepsilon(f)\dx + \lambda C_\varepsilon,
			\end{align}
		 	for a constant $C_\varepsilon\in (0, \infty)$, which is independent of $\lambda \in (0,\lambda_0)$.
		\item We show that for every $f\in L^1_\Phi(\rd) \cap L^\infty(\rd)$ with $f\geq 0$ a.e. and $\lambda \in (0,\lambda_0)$, 
		\begin{align}\label{theorem.FPKE.regularity.l1linfty:regularity.beta12:1.2}
			\int_\rd \Psi(u)\dx + \lambda\int_\rd |\nabla \beta^{\frac{1}{2}}(u)|^2\dx\leq\int_\rd \Psi(f)\dx + \lambda C,
		\end{align} where $u:=J_\lambda(f)$ (cf. Lemma \ref{BR21.lemma3.1}) for a constant $C>0$, which is independent of $\varepsilon$ and $\lambda \in (0,\lambda_0)$.
		\item We conclude \eqref{theorem.FPKE.regularity.l1linfty:regularity.beta12} via the finite difference scheme \eqref{theorem.FPKE.regularity.l1linfty:A.finite.difference.scheme}.
	\end{enumerate}

\noindent\textbf{Regarding 1.}	We use a similar ansatz as in the proof of \cite[Prop 5.6 (arXiv version 3)]{gess2023inventiones} and adapt their techniques to the mild solution framework in $L^1(\rd)$ in our setting.
	Let $\varepsilon \in (0,1)$. By Lemma \ref{lemma.FPKE.A_epsilon} we have that $u_\varepsilon\geq 0$ a.e., $u_\varepsilon \in H^1(\rd)\cap L^1(\rd)\cap L^\infty(\rd)$, $\tilde{\beta}_\varepsilon(u_\varepsilon) \in H^2(\rd)$, and $||u_\varepsilon||_{L^\infty} \leq  C_{E}||f||_{L^\infty}$, where $C_{E}= \left(1+ \norm[L^\infty(\rd)]{(\divv E)^- + |E|}^\frac{1}{2}\right)$.
	Here, note that 
	\begin{align*}
		\divv(E_\varepsilon b_\varepsilon(u_\varepsilon)u_\varepsilon) = \divv(E_\varepsilon)b_\varepsilon(u_\varepsilon)u_\varepsilon+ (b'_\varepsilon(u_\varepsilon) u_\varepsilon+b_\varepsilon(u_\varepsilon))\langle E_\varepsilon, \nabla u_\varepsilon\rangle_\rd \in L^1(\rd)\cap L^2(\rd).
	\end{align*}

	Let $\kappa_{n} \in C_c^\infty(\rd)$ such that $0\leq \kappa_n \leq 1$, $\kappa_n = 1$ on $B_n(0)$ and $||\nabla\kappa_n||_{L^\infty}+ ||\Delta\kappa_n||_{L^\infty} \leq C$ for some constant $C>0$ independent of $n$.
	We define $l_M(r):= ((\ln(r)\vee -M)\wedge M)\mathbbm{1}_{(0,\infty)}(r)-M\mathbbm{1}_{(-\infty,0]}(r)$ and $l_{M,n}(x,r):=\kappa_n(x)l_M(r), (r,x)\in \RR\times\rd$.
Note that 
 $l_M \in \Lip(\RR)$ with derivative $l_M' (r)= \mathbbm{1}_{(\exp{-M},\exp{M})}(r) \frac{1}{r}$ for all $r\notin\{\exp-M, \exp M\}$.
Let us multiply \eqref{theorem.FPKE.regularity.l1linfty:A_epsilon.eq} by $\lognbeta (\in L^1(\rd)\cap L^\infty(\rd)\cap H^1(\rd))$ and integrate over $\rd$. We obtain
\begin{align*}
		\underbrace{\int  \uepsiplus \lognbeta\dx}_{=:\text{I}} + \lambda \underbrace{\int - \Delta \tildebeta \lognbeta\dx}_{=:\text{II}} &+\lambda \varepsilon\underbrace{\int \tildebeta \lognbeta\dx}_{=:\text{III}} \\
		+ \lambda \underbrace{\int \divv(E_\varepsilon b_\varepsilon(\uepsiplus)\uepsiplus) \lognbeta\dx}_{=:\text{IV}}
		&= \underbrace{\int \uepsi \lognbeta\dx}_{=:\text{V}}.
	\end{align*}
	
\noindent \textbf{I}: By Lebesgue's dominated convergence theorem, it is standard to prove that
\begin{align*}
		\lim_{n\to \infty}\int  u_\varepsilon \lognbeta\dx = \int  u_\varepsilon l_M(\tilde{\beta}_{\varepsilon}(\uepsiplus))\dx.
	\end{align*}
	
\noindent \textbf{II}: Using integration by parts, we obtain
	\begin{align*}
		-&\int \Delta \tildebeta \lognbeta\dx
		= \int \langle\nabla \tildebeta, \nabla \lognbeta\rangle_\rd\dx\\
		&= \int |\nabla \tildebeta|^2 l_M'(\tildebeta) \kappa_n\dx
			+ \int \langle \nabla \tildebeta, \nabla \kappa_n\rangle_\rd l_M(\tildebeta)\dx
	\end{align*}
	Due to our choice of the $\kappa_n, n\in \LN$, and the fact that 
	$|\nabla \kappa_n|, \Delta \kappa_n \to 0$ on $\rd$, as $n\to \infty$, we may use integration by parts and Lebesgue's dominated convergence theorem to conclude that
	\begin{align*}
		\lim_{n\to \infty}\int \langle \nabla \tildebeta, \nabla \kappa_n\rangle_\rd l_M(\tildebeta)\dx
		&= \lim_{n\to \infty}-\int \tildebeta \left(l_M(\tildebeta)\Delta\kappa_n +\langle \nabla \kappa_n,\nabla\tildebeta\rangle_\rd l_M'(\tildebeta)\right)\dx \\
		&=0.
	\end{align*}
	Also, via Lebesgue's dominated convergence theorem, we have
	\begin{align*}
		\lim_{n\to \infty}&\int |\nabla \tildebeta|^2l_M'(\tildebeta) \kappa_n\dx 
		= \int |\nabla \tildebeta|^2 {l_M'(}\tildebeta)\dx.
	\end{align*}
	Using the monotone convergence theorem, we may, thus, conclude that
	\begin{align*}
		\lim_{M\to \infty}\lim_{n\to \infty}- &\int \Delta \tildebeta \lognbeta\dx
		= \int_{\{\tildebeta >0\}} \frac{|\nabla \tildebeta|^2}{\tildebeta}\dx.
	\end{align*}
\noindent \textbf{III:}
Similar to \cite[(4.5) and below]{barbu2021evolution} (and the references therein) we use that 
	\begin{align}\label{theorem.FPKE.regularity.l1linfty:estimate.rln-r}
		\forall \gamma \in (0,1)\ \exists C_\gamma>0\ \forall r\in [0,\infty): r\ln^-(r) \leq C_\gamma r^\gamma,
	\end{align}
	and  estimate for all $\gamma \in \left[\frac{\alpha}{\alpha+1},1\right)$\begin{align*}
	\lim_{n\to \infty} &\int \tildebeta\lognbeta\dx 
	= \int \tildebeta l_M(\tildebeta)\dx\\
	&\geq - \int \tildebeta l_M^-(\tildebeta)\dx
	\geq - \int \tildebeta \ln^-(\tildebeta)\dx
	\geq - C_{\gamma}\int \tildebeta^{\gamma}\dx\\
	&\geq -C_{\gamma}  \left(\int |\tildebeta| \Phi\dx\right)^{\gamma} \left(\int \Phi^{-\frac{\gamma}{1-\gamma}}\dx\right)^{1-\gamma}
	\geq -C_{\gamma}  \left(\int |\tildebeta| \Phi\dx\right)^{\gamma} \left(\int \Phi^{-\alpha}\dx\right)^{1-\gamma}\\
	&\geq - C_{\gamma}\left(\sup_{|r|\leq C_{E}||f||_{L^\infty}}\beta'(r) + 1\right)^{\gamma}\left(\sup_{\varepsilon\in (0,1)}||u_\varepsilon||_{L^1_{\Phi}} \right)^{\gamma} ||\Phi^{-\alpha}||_{L^1}^{1-\gamma} \\
	&\geq -C_{\Phi,\gamma},
\end{align*}
where we choose a finite nonnegative constant $C_{\Phi,\gamma}$ independent of $\varepsilon \in (0,1)$ and $\lambda \in (0,\lambda_0)$, which is possible due to Lemma \ref{lemma.BR21.Lemma3.3.proof}. 
In the above calculation, we used Lebesgue's dominated convergence theorem in the first equality and H\"older's inequality in the third line.
\\

\noindent \textbf{IV:} 
We set
$h_\varepsilon(r):=\frac{b_\varepsilon(r)r\tilde{\beta}_\varepsilon'(r)}{\tilde{\beta}_\varepsilon(r)}$, for $r\in (0,\infty)$.
Since $\beta \in C^2(\RR)$, $h_\varepsilon$ can be continuously extended to a function on $[0,\infty)$ by l'Hospital's rule. We denote this extended function also by $h_\varepsilon$.
Furthermore,
by \eqref{theorem.FPKE.regularity.l1linfty:cond.local}, we obtain that, via the mean-value theorem,
for every $K>0$ and $C^1_K:= (C_K (\sup_{r\in [0,K]} \beta'(r)+1))\vee 1$
\begin{align*}
	\beta'(g_\varepsilon(r))r\leq C_K^1\beta(g_\varepsilon(r))\ \ \forall r\in  [0,K].
\end{align*}
Hence, for all $r\in (0,K]$
\begin{align}\label{theorem.FPKE.regularity.l1linfty:lemma.estimate.hEps}
	(0 \leq)\ h_\varepsilon(r) \leq b_\varepsilon(r) \frac{\beta'(g_\varepsilon(r))r +\varepsilon r}{\beta(g_\varepsilon(r)) +\varepsilon r} 
	\leq C_K^1 ||b||_{L^\infty} <\infty.
\end{align}

\noindent We set $H_\varepsilon(r):=\int_0^r h_\varepsilon (s)ds, r\geq 0$. Note that $H_\varepsilon\geq 0$.
Using integration by parts, we obtain
\begin{align*}
	\int \divv(E_\varepsilon b_\varepsilon(u_\varepsilon)u_\varepsilon) \lognbeta\dx
	=-& \int \langle E_\varepsilon,\nabla \tildebeta\rangle_\rd b_\varepsilon(u_\varepsilon)u_\varepsilon l_M'(\tildebeta)\kappa_n\dx \\
	&- \int \langle E_\varepsilon, \nabla \kappa_n\rangle_\rd b_\varepsilon(u_\varepsilon)u_\varepsilon l_M(\tildebeta)\dx.
\end{align*}
By Lebesgue's dominated convergence theorem, 
the second summand on the right hand-side vanishes, as $n \to \infty$. Taking the limit of the first summand on the right hand-side yields with the aid of Lebegue's dominated convergence theorem
\begin{align*}
	\lim_{n\to \infty}&- \int \langle E_\varepsilon,\nabla \tildebeta\rangle_\rd b_\varepsilon(u_\varepsilon)u_\varepsilon l_M'(\tildebeta)\kappa_n\dx\\
	&=- \int \langle E_\varepsilon,\nabla \tildebeta\rangle_\rd b_\varepsilon(u_\varepsilon)u_\varepsilon l_M'(\tildebeta)\dx
	= - \int \langle E_\varepsilon, \nabla u_\varepsilon\rangle_\rd h_\varepsilon(u_\varepsilon)\mathbbm{1}_{(\exp{-M},\exp{M})}(\tildebeta)\dx.
\end{align*}
Using \eqref{theorem.FPKE.regularity.l1linfty:lemma.estimate.hEps} with $K:= C_E||f||_{L^\infty}$, Lebesgue's dominated convergence theorem and \eqref{FPKE.approx.divEeps-.estimate} we conclude that
\begin{align}
- \lim_{M\to \infty}&\int \langle E_\varepsilon, \nabla u_\varepsilon\rangle_\rd h_\varepsilon(u_\varepsilon)\mathbbm{1}_{(\exp{-M},\exp{M})}(\tildebeta)\dx
=- \int \langle E_\varepsilon, \nabla u_\varepsilon\rangle_\rd h_\varepsilon(u_\varepsilon)\dx\label{theorem.MVSDE.PU:ineq.uhEta2.div:1}\\
&= \int (\divv E_\varepsilon) H_\varepsilon(u_\varepsilon)\dx
\geq -\int (\divv E_\varepsilon)^- H_\varepsilon(u_\varepsilon)\dx \label{theorem.MVSDE.PU:ineq.uhEta2.div:2}\\
&\geq - \int \left[(\divv E)^- + |E|\mathbbm{1}_{\{|x| \geq \varepsilon\inv\}}\right]H_\varepsilon(u_\varepsilon)\dx\\
&\geq- C_K^1||(\divv E)^- + |E|||_{L^\infty}
||b||_{L^\infty(\rd)}||f||_{L^1(\rd)} \\
&=: -C_{\divv} > -\infty.
\end{align}
Note that in the first equality we used that $\nabla u_\varepsilon = 0$ a.e. on the set $\{\tildebeta =0 \}= \{u_\varepsilon=0\}$. Assuming \eqref{theorem.FPKE.regularity.l1linfty:cond.E} instead of \eqref{theorem.FPKE.regularity.l1linfty:cond.local}, one can choose $C_{\text{div}}:=0$, which is legitimated by $E=E_\varepsilon$, \eqref{theorem.MVSDE.PU:ineq.uhEta2.div:1}-\eqref{theorem.MVSDE.PU:ineq.uhEta2.div:2}, and the nonnegativity of $H_\varepsilon(u_\varepsilon)$. \\

\noindent \textbf{V}:	
Similar to the argument in I, we conclude
\begin{align*}
		\lim_{n\to \infty} \int  \uepsi \lognbeta\dx = \int  \uepsi l_M(\tilde{\beta}_{\varepsilon}(\uepsiplus))\dx.
	\end{align*}
\vspace{2em}	
\\
	\indent Now we prove \eqref{theorem.FPKE.regularity.l1linfty:regularity.beta12:1.1}.
	Since $l_M(\tilde{\beta}_\varepsilon)$ is non-decreasing, we have for $\Psi^M_\varepsilon(r):= \int_0^r l_M(\tilde{\beta}_{\varepsilon}(s))ds, r\geq 0$,
	\begin{align*}
		\Psi^M_\varepsilon(r_1) - \Psi^M_\varepsilon(r_2) \leq \l_M(\tilde{\beta}_\varepsilon(r_1)) (r_1-r_2) \ \forall r_1,r_2 \in \RR.
	\end{align*}
Hence, combining the limits and estimates we obtained for \textbf{I} -\textbf{V}, we obtain for $M\to \infty$ (at least formally up to this point)
\begin{align}\label{theorem.FPKE.regularity.l1linfty:theorem.FPKE.regularity.l1linfty:inequality.epsilon.M}
	\liminf_{M\to \infty}\int \Psi^M_\varepsilon(\uepsiplus)\dx \ +&\ \lambda \int_{\{\tildebeta >0\}} \frac{|\nabla \tildebeta|^2}{\tildebeta}
	dx - \lambda(\varepsilon C_{\Phi,\gamma}+C_{\divv})\notag\\
	&\leq \liminf_{M\to \infty}\int \Psi^M_\varepsilon(\uepsi)\dx.
\end{align}

\noindent Now, let us verify that the limites inferiores in \eqref{theorem.FPKE.regularity.l1linfty:theorem.FPKE.regularity.l1linfty:inequality.epsilon.M} are actually real limits (which a-posteori makes \eqref{theorem.FPKE.regularity.l1linfty:theorem.FPKE.regularity.l1linfty:inequality.epsilon.M} a rigorous statement).
Note that
\begin{align}\label{psiMeps.1}
	|l_M(r)| 
	= (\ln^+(r)\wedge M) + (\ln^-(r)\wedge M) \ \ \forall r\in [0,\infty).
\end{align}
Hence, $|\Psi_\varepsilon^M| \leq \int_0^\cdot|\ln(\tilde{\beta}_\varepsilon(r))|dr$ on $[0,\infty)$. 
Let us show that $\int_0^\cdot|\ln(\tilde{\beta}_\varepsilon(r))|dr \in L^1(\rd)$; it will be even shown that $(\int_0^\cdot|\ln(\tilde{\beta}_\varepsilon(r))|dr)_{\varepsilon \in (0,1)}$ is bounded in $L^1(\rd)$.
Since $\ln^+$ is Lipschitz continuous with $\ln^+(0)=0$ and Lipschitz constant $1$, we may estimate
\begin{align}\label{theorem.FPKE.regularity.l1linfty:psiMeps.2}
	\int_0^{u_\varepsilon(x)}\ln^+(\tilde{\beta}_\varepsilon(r)) dr \leq \tildebeta u_\varepsilon(x).
\end{align}
Hence, recalling that $||u_\varepsilon||_{L^1} \leq ||f||_{L^1}, ||u_\varepsilon||_{L^\infty} \leq C_{E}||f||_{L^\infty}$, we estimate
\begin{align*}
	\int_\rd \int_0^{u_\varepsilon(x)}\ln^+(\tilde{\beta}_\varepsilon(r)) \dr\dx 
	\leq \left(\sup_{|r|\leq C_{E}||f||_{L^\infty}} \beta'(r) +1\right) C_{E}||f||_{L^\infty(\rd)}||f||_{L^1(\rd)}<\infty,
\end{align*}
Moreover, by the transformation rule for the Lebesgue integral, we obtain
\begin{align}\label{theorem.FPKE.regularity.l1linfty:psiMeps.3}
	\int_0^{u_\varepsilon(x)} \ln^-(\tilde{\beta}_\varepsilon(r)) \dr
	\leq \int_0^{u_\varepsilon(x)}\ln^-(\beta(g_\varepsilon(r)))\dr
	= \int_0^{g_\varepsilon(u_\varepsilon(x))}\ln^-(\beta(r)) (1+\varepsilon\beta'(r))\dr.
\end{align}
Note that \ref{condition.a.general} implies that $\ln^-(\beta(r)) \leq\ln^-\left(\frac{a_1}{m_1}\right)+ m_1 \ln^-(r)$ for all $r> 0$. Similar to the estimates we obtained for \textbf{III}, noting that $g_\varepsilon(r) \leq r\ \forall r\geq 0$, we find that for all $\gamma \in \left[ \frac{\alpha}{\alpha+1},1\right)$ and $C_{\beta'}:= \left(\sup_{|r|\leq C_{E}||f||_{L^\infty}}\beta'(r) +1\right)$
\begin{align}\label{theorem.FPKE.regularity.l1linfty:psiMeps.4}
	\int_\rd &\int_0^{g_\varepsilon(u_\varepsilon(x))}\ln^-(\beta(r)) (1+\varepsilon\beta'(r))\dr\dx \notag\\
	&\leq C_{\beta'}\left[m_1\int_\rd \int_0^{u_\varepsilon(x)}\ln^-(r)drdx+\ln^-\left(\frac{a_1}{m_1}\right)||u_\varepsilon||_{L^1}\right]\notag \\
	&= C_{\beta'}\left[m_1\int_\rd u_\varepsilon(x)\ln^-(u_\varepsilon(x))+ (u_\varepsilon(x)\wedge 1) \dx + \ln^-\left(\frac{a_1}{m_1}\right)||u_\varepsilon||_{L^1}\right]\notag\\
	&\leq C_{\beta'} 
		\left[m_1C_{\gamma} ||u_\varepsilon||_{L^1_\Phi}^{\gamma}||\Phi^{-\alpha}||_{L^1}^{1-\gamma} + \left(\ln^-\left(\frac{a_1}{m_1}\right)+m_1\right)||u_\varepsilon||_{L^1}\right]\notag\\
	&\leq C_{\beta'} 
		\left[m_1C_{\gamma} \sup_{\varepsilon\in (0,1)}||u_\varepsilon||_{L^1_\Phi}^{\gamma}||\Phi^{-\alpha}||_{L^1}^{1-\gamma} + \left(\ln^-\left(\frac{a_1}{m_1}\right)+m_1\right)||f||_{L^1}\right]<\infty.
\end{align}
Since $\Psi_\varepsilon^M(r) \to \Psi_\varepsilon(r)$ for all $r \in [0,\infty)$, we therefore conclude that, by Lebesgue's dominated convergence theorem,
\begin{align*}
	\Psi_\varepsilon^M(u_\varepsilon) \to \Psi_\varepsilon (u_\varepsilon) \text{ in } L^1(\rd), \text{ as } {M\to \infty},
\end{align*}
An analogous argument yields 
\begin{align*}
	\Psi_\varepsilon^M(f) \to \Psi_\varepsilon (f) \text{ in } L^1(\rd), \text{ as } {M\to \infty}.
\end{align*}
Consequently, \eqref{theorem.FPKE.regularity.l1linfty:theorem.FPKE.regularity.l1linfty:inequality.epsilon.M} is equivalent to 
\begin{align}\label{theorem.FPKE.regularity.l1linfty:inequality.epsilon}
	\int \Psi_\varepsilon(\uepsiplus) \dx \ +&\ \lambda \int_{\{\tildebeta >0\}} \frac{|\nabla \tildebeta|^2}{\tildebeta}\dx - \lambda(\varepsilon C_{\Phi,\gamma}+C_{\divv})\notag\\
	&\leq \int \Psi_\varepsilon(\uepsi) \dx.
\end{align}

Let $\varphi \in C_c^\infty(\rd)$ and $\delta>0$ be arbitrary but fixed. Then, integrating by parts, we obtain
\begin{align}\label{theorem.FPKE.regularity.l1linfty:beta12weakDerivative}
	\int (\tildebeta + \delta)^\frac{1}{2}\nabla \varphi\dx
	= - \int \frac{\nabla \tildebeta}{2(\tildebeta + \delta)^\frac{1}{2}}\varphi\dx
	= - \int_{\{\tildebeta >0\}} \frac{\nabla \tildebeta}{2(\tildebeta + \delta)^\frac{1}{2}}\varphi\dx.
\end{align}
where the last equality follows from the fact that $\nabla \tildebeta = 0$ a.e. on $\{\tildebeta =0\}$.
By Lebesgue's dominated convergence theorem, we may take the limit $\delta \to 0$ in \eqref{theorem.FPKE.regularity.l1linfty:beta12weakDerivative} and, hence, obtain from \eqref{theorem.FPKE.regularity.l1linfty:inequality.epsilon} that $\tilde{\beta}_\varepsilon^\frac{1}{2}(u_\varepsilon) \in H^1(\rd)$ with 
$$\nabla \tilde{\beta}_\varepsilon^\frac{1}{2}(u_\varepsilon) = \mathbbm{1}_{\{\tildebeta >0\}} (2\tilde{\beta}_\varepsilon^\frac{1}{2}(u_\varepsilon))\inv\nabla \tildebeta.$$
Consequently, \eqref{theorem.FPKE.regularity.l1linfty:regularity.beta12:1.1} becomes evident.\\

\noindent\textbf{Regarding 2.} 
Let $f \in L^1_\Phi(\rd)\cap L^\infty(\rd)$ such that $f\geq 0$ a.e. Then, by Lemma \ref{BR21.lemma3.1}, 
 $u_\varepsilon(\lambda,f) \to J_\lambda (f)=:u$ in $L^1(\rd)$ and therefore also pointwise almost everywhere for possibly a subsequence, which is not denoted differently.
By the proof of Step 1, 
we have that $(\nabla\tilde{\beta}_{\varepsilon}^{\frac{1}{2}}(u_\varepsilon))_{\varepsilon \in (0,1)}$ is bounded in $L^2(\rd)$, 
as each of the families $\left(\int \Psi_\varepsilon(u_\varepsilon)\dx\right)_{\varepsilon\in (0,1)}, \left(\int \Psi_\varepsilon(f)\dx\right)_{\varepsilon\in (0,1)}$ is bounded.
Note that $(\tilde{\beta}_\varepsilon)_{\varepsilon\in (0,1)}$ is locally equicontinuous, whence $\tildebeta \to \beta(u)$ a.e. Since, for all $\varepsilon\in (0,1)$, $\tildebeta\leq \left(1+\sup_{|r|\leq C_{E}||f||_{L^\infty}}\beta'(r)\right)f\ (\in L^1(\rd)\cap L^\infty(\rd))$, $\tilde{\beta}_\varepsilon(u_\varepsilon) \to \beta(u)$ in $L^p(\rd)$, as $\varepsilon\to 0$, for all $p\geq 1$.
Hence, via the Banach--Alaoglu theorem, it is easy to see that there exists a (non-relabled) subsubsequence such that
\begin{align}\label{theorem.FPKE.regularity.l1linfty:beta12.deriv.weakConvergence}
	\nabla\tilde{\beta}_{\varepsilon}^{\frac{1}{2}}(u_\varepsilon) \to \nabla\beta^{\frac{1}{2}}(u) \text{ weakly in } L^2(\rd), \text{ as } \varepsilon \to 0.
\end{align}

Note that for all $\gamma \in \left(\frac{\alpha}{\alpha + 1},1\right)$
\begin{align*}
	\int u_\varepsilon^\gamma\dx \to \int u^\gamma\dx, \text{ as } \varepsilon\to 0.
\end{align*}
Indeed, for all $\rho \in (0,1)$ with $\gamma \in \left[\frac{\alpha}{\alpha + (1-\rho)},1\right)$ we have
\begin{align*}
	\left|\int u_\varepsilon^\gamma - u^\gamma\dx\right| \leq ||u_\varepsilon-u||_{L^1}^{\gamma \rho} ||u_\varepsilon - u||_{L^1_{\Phi}}^{\gamma(1-\rho)}||\Phi^{-\alpha}||_{L^1}^{1-\gamma}.
\end{align*}
Since by \eqref{lemma.FPKE.A_epsilon:convergenceJlambda} and Lemma \ref{lemma.BR21.Lemma3.3.proof}, $||u||_{L^1_{\Phi}}+\sup_{\varepsilon\in (0,1)}||u_\varepsilon||_{L^1_\Phi} <\infty$ and $u_\varepsilon \to u$ in $L^1(\rd)$, the assertion is proved.
Furthermore, by the generalised Lebesgue's dominated convergence theorem, it is straight-forward to see that $\Psi_\varepsilon(u_\varepsilon)\to \Psi(u)$ and $\Psi_\varepsilon(f)\to \Psi(f)$ a.e.
Hence, using \eqref{theorem.FPKE.regularity.l1linfty:estimate.rln-r} and that for
all $\gamma \in (0,1)$
 \begin{align*}
 	|\Psi_\varepsilon(u_\varepsilon)| \leq \left(\sup_{|r|\leq C_E||f||_{L^\infty}}\beta'(r) +1\right) u_\varepsilon^2 + \ln^-\left(\frac{a_1}{m}\right)u_\varepsilon + m\left(\sup_{|r|\leq C_E||f||_{L^\infty}}\beta'(r) +1\right)\left[C_{\gamma}u_\varepsilon^\gamma + u_\varepsilon\right],
 \end{align*}
 we conclude, via the generalised Lebesgue's dominated convergence theorem, that
 \begin{align}\label{theorem.FPKE.regularity.l1linfty:psiEps(ueps)ToPsi(u)}
 	\int \Psi_\varepsilon(u_\varepsilon)\dx\to \int \Psi(u)\dx, \text{ as } \varepsilon\to 0.
 \end{align}
Similarly,  
  \begin{align*}
 	\int \Psi_\varepsilon(f)\dx\to \int \Psi(f)\dx, \text{ as } \varepsilon\to 0.
 \end{align*}
 
Hence, taking $\liminf_{\varepsilon\to 0}$ on both sides in \eqref{theorem.FPKE.regularity.l1linfty:inequality.epsilon}, we obtain
\begin{align}\label{theorem.FPKE.regularity.l1linfty:inequality.final}
	\int \Psi(u)\dx \ + 4\lambda \int |\nabla \tilde{\beta}_{}^{\frac{1}{2}}(u)|^2dx - \lambda C_{\divv}
	\leq \int \Psi(f)\dx.
\end{align}
This finishes the proof for \eqref{theorem.FPKE.regularity.l1linfty:regularity.beta12:1.2}.

\noindent \textbf{Regarding 3.}
Now let us consider the finite difference scheme \eqref{theorem.FPKE.regularity.l1linfty:A.finite.difference.scheme}. 
Due to \eqref{theorem.FPKE.regularity.l1linfty:inequality.final}, for all $h \in (0,\gamma_0)$ and $i\in \{0,...,N-1\}$, we obtain the estimate
\begin{align}\label{theorem.FPKE.regularity.l1linfty:inequality.final.differenceScheme}
	\int \Psi(u_{h}^{i+1})\dx \ +&\ 4h \int |\nabla\beta^{\frac{1}{2}}(u_{h}^{i+1})|^2
	\dx - hC_{\divv}
	\leq \int \Psi (u_{h}^{i})\dx.
\end{align}
Let $t\in (0,T)$ and note that $t \leq \left[\frac{Nt}{T}\right]h +h<T$.
Now, summing from $0$ to $k-1:= \left[\frac{Nt}{T}\right]$, we obtain
\begin{align*}
	\int \Psi(u_{h}^{k})\dx \ +&\ 4h\sum_{i=0}^{k-1} \int |\nabla\beta^{\frac{1}{2}}(u_{h}^{i+1})|^2\dx 
	- TC_{\divv}
	\leq \int \Psi (\nu)\dx.
\end{align*}
Using \eqref{theorem.FPKE.regularity.l1linfty:A.finite.difference.scheme}, we may conclude that
\begin{align}\label{theorem.FPKE.regularity.l1linfty:inequality.final.differenceScheme.u_h(t)}
	\int \Psi(u_{h}(t))\dx \ +&\ 4\int_0^{\left[\frac{Nt}{T}\right]h +h} \int 
	|\nabla \beta^{\frac{1}{2}}(u_{h}(s))|^2\dx\ds 
	- TC_{\divv}
	\leq \int \Psi(\nu)\dx.
\end{align}
By \cite[(3.64) and below]{barbu2020solutions}, there exists $h_0 \in (0,\gamma_0)$ such that for all $0<h<h_0$
\begin{align}\label{theorem.FPKE.existence:Linfty:u_h}
		\norm[L^\infty]{u_h(t)} \leq \exp\left(\norm[L^\infty]{(\divv E)^-+|E|}^\frac{1}{2}t\right)\norm[L^\infty]{\nu}+1 \ \forall t\in [0,\infty).
	\end{align}
Let us set $C_{\nu}:=\exp\left(\norm[L^\infty]{(\divv E)^-+|E|}^\frac{1}{2}T\right)\norm[L^\infty]{\nu}+1$.

Furthermore, by Remark \ref{lemma.BR21.Lemma3.3.proof:remark}, \cite[below (3.64)]{barbu2020solutions} and \eqref{theorem.FPKE.regularity.l1linfty:A.finite.difference.scheme},
 $$\sup_{t\in (0,T), h \in (0,h_0)}||u_h(t)||_{L^1_{\Phi}}
 \leq ||\nu||_{L^1_\Phi} + T\left[\left(\sup_{|r|\leq C_{\nu}} \beta'(r)\right) ||\Delta \Phi||_{L^\infty} + ||E||_{L^\infty}||b||_{L^\infty}||\nabla\Phi||_{L^\infty}\right]||\nu||_{L^1}.$$
  Therefore, a similar reasoning as for \eqref{theorem.FPKE.regularity.l1linfty:psiEps(ueps)ToPsi(u)} yields that (for a non-relabled subsequence)
\begin{align*}
	\int\Psi(u_h(t))\dx \to \int\Psi(u(t))\dx, \text{ as } h \to 0,
\end{align*}
where the convergence holds in $\RR$.

Hence,
$(\nabla \beta^{\frac{1}{2}}(u_h))_{h \in (0,h_0)}$ is bounded in $L^2((0,t)\times\rd)$ for all $t\in (0,T)$. It is easy to check that, for all $p\geq 2$, $\beta^\frac{1}{2}(u_h) \to \beta^{\frac{1}{2}}(u)$ in $L^p((0,T)\times\rd)$ for a subsequence $h\to 0$. Thus, it is straight-forward to see that via the Banach--Alaoglu theorem, for a non-relabled subsubsequence, 
\begin{align*}
	\nabla\beta^{\frac{1}{2}}(u_h) \to \nabla\beta^{\frac{1}{2}}(u) \text{ weakly in } L^2((0,t)\times\rd), \text{ as } h\to 0.
\end{align*}
Hence, taking $\liminf_{h\to 0}$ in \eqref{theorem.FPKE.regularity.l1linfty:inequality.final.differenceScheme.u_h(t)}, we obtain for all $t\in (0,T)$
\begin{align*}
	\int \Psi(u(t))\dx \ +
	 4\int_0^{t} \int 
	 |\nabla \beta^{\frac{1}{2}}(u_{}(s))|^2\dx\ds 
	 - TC_{\divv}
	 \leq \int \Psi (\nu) \dx.
\end{align*}
This completes the proof.
\end{proof}
In order to prove Theorem \eqref{theorem.mainresult.1}, we will use Theorem \ref{theorem.restrictedYamadaWatanabe} and show (i) in the latter's formulation.
This will be done in the subsequent two subsections.
\subsection{Ingredient 1: A $P^{(u_t)}_\nu$-weak solution to \eqref{MVSDE}}\label{subsection.theorem.mainresult.1.proof:ingredient1}
	
The following result is part of \cite[Theorem 6.1 (a)]{barbu2020solutions}, which is based on the superposition principle procedure from \cite[Section 2]{barbu2019nonlinear}.
\begin{theorem}[$P^{(u_t)}_\nu$-weak solution]\label{theorem.MVSDE.existence}
	Assume that \ref{condition.a.general.l1linfty} and \ref{condition.b.general.l1linfty} hold. Let $\nu \in \PPPP_0(\rd)\cap L^\infty(\rd)$. Then for each $T>0$, there exists a $P^{(u_t)}_\nu$-weak solution $(X,W)$ to \eqref{MVSDE} up to time $T$,
	where $u$ is the probability solution to \eqref{FPKE} with $\left.u\right|_{t=0}=\nu$ provided by Theorem \ref{theorem.FPKE.existence}.
\end{theorem}

\subsection{Ingredient 2: $P^{(u_t)}_\nu$-pathwise uniqueness for \eqref{MVSDE}}\label{subsection.theorem.mainresult.1.proof:ingredient2}
Suppose $u$ is a probability solution to \eqref{FPKE} with initial condition $\left.u\right|_{t=0}=\nu \in \PPPP(\rd)$.
Recall that, by definition, $P^{(u_t)}_\nu$-weak solutions to \eqref{MVSDE} all have the same time marginal laws. Equivalently, they are $P^{(u_t)}_\nu$-weak solutions to the following \textit{ordinary} SDE
\begin{align}\label{MVSDE.fixed.u}\tag{$\text{SDE}_u$}
	dX(t)&=\bm{b}^u(t,X(t))\ \mathrm dt + \bm{\sigma}^u(t,X(t))\ \mathrm dW(t),\ \ t\in [0,T],\\
	\law{X(0)} &= \nu,\notag
\end{align}
where $\bm{b}^u(t,x):=E(x)b(u_t(x))$ and $\bm{\sigma}^u(t,x):=\sqrt{2\frac{\beta(u_t(x))}{u_t(x)}}\mathbbm{1}_{d\times d}$, $(t,x)\in [0,T]\times\rd$.
As already indicated in the introduction, we will fix a precise representative of the $u_t$, $t\in [0,T]$, in the coefficients of \eqref{MVSDE} such that $\bm{b}^u$, $\bm{\sigma}^u$ are Borel measurable. In particular, $u$ is pointwise defined on $[0,T]\times \rd$.
In the following theorem, we are going to prove $P^{(u_t)}_\nu$-pathwise uniqueness for \eqref{MVSDE.fixed.u}.
The proof is based on a combination of a proper modification of a pathwise uniqueness result for SDEs performed in the proof of \cite[Theorem 1.1]{roeckner2010weakuniqueness}, Theorem \ref{theorem.MVSDE.PU}, and a trick involving the mean-value theorem, which is motivated by \cite[(76)-(77)]{gess2023inventiones}.

\begin{theorem}[$P^{(u_t)}_\nu$-pathwise uniqueness]\label{theorem.MVSDE.PU}
	Assume that all the conditions in Theorem \ref{theorem.FPKE.regularity.l1linfty} are fulfilled. Furthermore, let $m \in (1,\infty), \zeta \in [0,1]$ such that $\frac{2\zeta}{m}<1$ and assume that
	\begin{align}
		\label{theorem.MVSDE.PU:ii} \forall K>0 \ \exists C_K,a_K\in (0,\infty)\ \forall r\in [0,K]: \beta(r) \leq C_Kr^{m} \text{ and }a_K|r|^{m-1}\leq \beta'(r),
	\end{align}
	that for all $R>0$ there exits $\iota_R \in L^1(B_R(0))$ such that for a.e. $x, y \in B_R(0)$
	\begin{align}
		\langle E(x)-E(y),x-y\rangle_{\rd} \leq (\iota_R(x) + \iota_R(y))|x-y|_{\rd}^2,
	\end{align}
	and that
	$b\circ (G\circ\beta^\frac{1}{2})\inv \in \Lip_{loc}([0,\infty))$, where $G(r) := \int_0^r (\beta\inv (s^2))^{-\zeta}\ds$, for all $r\geq 0$.
	
	Let $\nu \in \PPPP_0(\rd)\cap L^\infty(\rd)\cap L^1_{\Phi}(\rd)$ for some $\Phi \in \mathcal{Q}$ and let $u$ denote the corresponding probability solution to \eqref{FPKE} with $\left.u\right|_{t=0}=\nu$ provided by Theorem \ref{theorem.FPKE.existence}.
Let $T\in (0,\infty)$. Suppose $(X,W), (Y,W)$ are two $P^{(u_t)}_\nu$-weak solutions to \eqref{MVSDE} up to time $T$ on a common stochastic basis $(\Omega,\FF,\PP;(\FF_t)_{t\in [0,T]})$ with respect to the same standard $d$-dimensional  $(\FF_t)$-Brownian motion $(W(t))_{t\in [0,T]}$ such that $X(0)=Y(0)$ $\PP$-a.s.
Then ${\sup_{t\in [0,T]} |X(t)-Y(t)| = 0\  \PP\text{-a.s.}}$
\end{theorem}
\begin{proof}
	Let $Z(t):=X(t)-Y(t), t\in [0,T]$, $\tau_R:=\inf \{ t\in [0,T] : |X(t)|\vee |Y(t)| \geq R\}\wedge T $  for $R>0$, and $\delta>0$ be arbitrary but fixed.
	As in the proof of \cite[Theorem 1.1]{roeckner2010weakuniqueness}, we apply It\^o's formula to $f_\delta(Z(t\wedge\tau_R)$, where $f_\delta$ is defined as (the $C^2$-function) $f_\delta(x):=\ln\left(\frac{|x|^2}{\delta^2}+1\right), x\in \rd,$ with $|\nabla f_\delta(x)| \leq \frac{C}{|x|+\delta}, |\Delta f_\delta(x)|\leq \frac{C}{|x|^2+\delta^2},$ for all $x\in \rd$ and some $C\in (0,\infty)$ independent of $x$.
	We obtain the following estimates, utilising Fatou's lemma.
	\begin{align}\label{theorem.MVSDE.PU:ineq}
		&\EE\ln\left(1+\frac{|Z(t\wedge\tau_R)|^2}{\delta^2}\right) \\
	&\leq C\EE\int_0^{t\wedge\tau_R}\frac{\langle E(X(s))b(u_s(X(s)))-E(Y(s))b(u_s(Y(s))),Z(s)\rangle_{\rd} }{|Z(s)|^2+\delta^2}\ds \notag\\
	&\ \ \ + C\EE\int_0^{t\wedge\tau_R}\frac{\left|\left(\frac{\beta(u_s(X(s))}{u_s(X(s))}\right)^{\frac{1}{2}}-\left(\frac{\beta(u_s(Y(s))}{u_s(Y(s))}\right)^{\frac{1}{2}}\right|^2}{|Z(s)|^2+\delta^2}\ds\notag\\
	&\leq C\EE\int_0^{t\wedge\tau_R}\frac{\langle E(X(s))b(u_s(X(s)))-E(Y(s))b(u_s(Y(s))),Z(s)\rangle_{\rd} }{|Z(s)|^2+\delta^2}\ds \notag\\
	&\ \ \ + C\liminf_{\eta\to 0} \EE\int_0^{t\wedge\tau_R}\frac{\left|\left(\frac{\beta(u_s(X(s))}{u_s(X(s))+\eta}\right)^{\frac{1}{2}}-\left(\frac{\beta(u_s(Y(s))}{u_s(Y(s))+\eta}\right)^{\frac{1}{2}}\right|^2}{|Z(s)|^2+\delta^2}\ds\notag\\
	&=: \text{I }+\liminf_{\eta \to 0}\text{ II$_\eta$}.
	\end{align}

	\noindent \textbf{Regarding II$_\eta$:}\\
	Let $\eta >0$ and define $f_\eta(r):=\frac{r}{(\beta\inv(r^2)+\eta)^\frac{1}{2}}, r \geq 0$. 
	Note that $f_\eta$ is continuously differentiable in a sufficiently small open neighbourhood of $[0,\infty)$ in $\RR$. A simple computation yields
	\begin{align}\label{theorem.MVSDE.PU:f.derivative} 
		f_\eta'(r)
		= \frac{1}{(\beta\inv(r^2)+\eta)^\frac{1}{2}}-\frac{r^2(\beta\inv)'(r^2)}{(\beta\inv(r^2)+\eta)^{\frac{3}{2}}} 
		\ \ \forall r\geq 0.
	\end{align}
	Let $\tilde{K} \in (0,\infty)$ and $K:=\beta\inv(\tilde{K})$. Then by  \eqref{theorem.MVSDE.PU:ii} and the inverse function theorem, one can find $C_1=C_1(\tilde{K}), C_2=C_2(\tilde{K})\in (0,\infty)$ such that
	\begin{align*}
		C_1 r^{\frac{1}{m}} \leq \beta\inv(r)\ \ \forall r\in [0,\tilde{K}] \text{ and } (\beta\inv)'(r) \leq C_2 r^{-\frac{m-1}{m}}\ \ \forall r\in (0,\tilde{K}].
	\end{align*}
	Via \eqref{theorem.MVSDE.PU:f.derivative}, we obtain for $C_3:= \frac{C_2}{C_1}+1$
	\begin{align*}
		|f_\eta'(r)| \leq C_3(C_1r^\frac{2}{m}+\eta)^{-\frac{1}{2}}\ \text{ $\forall r\in [0,\tilde{K}]$}.
	\end{align*}
	Therefore, choosing $\tilde{K}:=\beta^\frac{1}{2}(||u||_{L^\infty})\vee \beta(||u||_{L^\infty})$, we may estimate via the mean value theorem for all ${x,y \in \rd, s\in [0,T]}$
	\begin{align}\label{theorem.MVSDE.PU:trick1}
		&\left|\left(\frac{\beta(u_s(x))}{u_s(x)+\eta}\right)^{\frac{1}{2}}-\left(\frac{\beta(u_s(y))}{u_s(y)+\eta}\right)^{\frac{1}{2}}\right|
		=|f_\eta(\beta^\frac{1}{2}(u_s(x)))-f_\eta(\beta^\frac{1}{2}(u_s(y)))|\notag\\
		&\leq\int_0^1 \left|f_\eta'\left(\theta \beta^\frac{1}{2}(u_s(x))+(1-\theta)\beta^\frac{1}{2}(u_s(y))\right)\right|\mathrm d\theta \left|\beta^\frac{1}{2}(u_s(x))-\beta^\frac{1}{2}(u_s(y))\right|\notag\\
		&\leq h_\eta(s,x)\left|\beta^\frac{1}{2}(u_s(x))-\beta^\frac{1}{2}(u_s(y))\right|, 
	\end{align}
	where  $h_{\eta}(s,x) := C_3\int_0^1(C_4u_s(x)\theta^\frac{2}{m}+\eta)^{-\frac{1}{2}}\mathrm d\theta, s\in [0,T],x\in \rd$, and $C_4 := C_1(\frac{a}{m})^{\frac{1}{m}}$. Note that $||h_\eta||_\infty \leq \frac{C_3}{\sqrt{\eta}}$ and
	\begin{align}\label{theorem.MVSDE.PU:ineq.uhEta2}
		u_s(x) h_\eta^2(s,x) \leq C_3^2C_4^{-1}\left(\int_0^1 \theta^{-\frac{1}{m}}\mathrm d\theta\right)^2<\infty.
	\end{align}
	For $f\in L^1_{loc}(\rd)$, we set $f_\varepsilon(x):=(f\ast\varphi_\varepsilon)(x)$, for all $x\in \rd$, where $\varphi_\varepsilon(x)=\varepsilon^{-d}\varphi(\varepsilon\inv x)$, for some fixed $\varphi \in \PPPP_0(\rd)\cap C_c^\infty(\rd), \varepsilon\in (0,1)$.
	 Using this notation and \eqref{theorem.MVSDE.PU:trick1} we estimate
	 \begin{align}\label{theorem.MVSDE.PU:II}
	 	\text{II}_\eta&\leq C \EE \int_0^{t\wedge\tau_R} (h_\eta(s,X(s))\wedge h_\eta(s,Y(s)))^2\frac{\left|(\beta^\frac{1}{2}(u_s))_\varepsilon(X(s))-(\beta^\frac{1}{2}(u_s))_\varepsilon(Y(s))\right|^2}{|Z(s)|^2+\delta^2}\ds \\
	 	&+ C \EE  \int_0^{t\wedge\tau_R} (h_\eta(s,X(s))\wedge h_\eta(s,Y(s)))^2\frac{\left|(\beta^\frac{1}{2}(u_s))_\varepsilon(X(s))-\beta^\frac{1}{2}(u_s(X(s)))\right|^2}{|Z(s)|^2+\delta^2}\ds\notag\\
	 	&+ C \EE \int_0^{t\wedge\tau_R} (h_\eta(s,X(s))\wedge h_\eta(s,Y(s)))^2\frac{\left|(\beta^\frac{1}{2}(u_s ))_\varepsilon(Y(s))-\beta^\frac{1}{2}(u_s(Y(s)))\right|^2}{|Z(s)|^2+\delta^2}\ds\notag\\
	 	&=: \text{II}_1^{\varepsilon,\eta} + \text{II}_2^{\varepsilon,\eta} + \text{II}_3^{\varepsilon,\eta}.\notag
	 	\end{align}
	 	By Lemma \ref{appendix.lemma.lipschitztypeestimate}, we have for all $x,y \in \rd$
	 	\begin{align*}
	 		\left|(\beta^\frac{1}{2}(u_s))_\varepsilon(x)-(\beta^\frac{1}{2}(u_s))_\varepsilon(y)\right|
	 		\leq \left(\left(\M|\nabla \beta^\frac{1}{2}(u_s)|\right)_\varepsilon(x)+\left(\M|\nabla \beta^\frac{1}{2}(u_s)|\right)_\varepsilon(y)\right)|x-y| .
	 	\end{align*}
	 	Therefore, using Lemma \ref{appendix.lemma.boundednesslocalmaximalfunction} and \eqref{theorem.MVSDE.PU:ineq.uhEta2}, we estimate 
	 	\begin{align}\label{theorem.MVSDE.PU:II.1}
	 	\text{II}_1^{\varepsilon,\eta}
	 	&\leq  C \EE\int_0^{t\wedge\tau_R} (h_\eta(s,X(s))\wedge h_\eta(s,Y(s)))^2\frac{\left(\left(\M|\nabla \beta^\frac{1}{2}(u_s)|\right)_\varepsilon(X(s))+ \left(\M|\nabla \beta^\frac{1}{2}(u_s)|\right)_\varepsilon(Y(s))\right)^2 |Z(s)|^2}{|Z(s)|^2+\delta^2}\ds \notag\\
	 	&\leq   C \int_0^T \int_\rd \left(\left(\M|\nabla \beta^\frac{1}{2}(u_s)|\right)_\varepsilon(x)\right)^2 h_\eta^2(s,x)u_s(x)\dx\ds
	 	\leq C ||\nabla\beta^{\frac{1}{2}}(u)||_{L^2((0,T)\times\rd;\rd)}^2<\infty,
	 \end{align}
	 where the right hand-side is finite, according to Theorem \ref{theorem.FPKE.regularity.l1linfty}.
	 Furthermore, using \eqref{theorem.MVSDE.PU:ineq.uhEta2}, we obtain
	 \begin{align}\label{theorem.MVSDE.PU:II.2}
	 	\max\{\text{II}_2^{\varepsilon,\eta}, \text{II}_2^{\varepsilon,\eta} \} &\leq \frac{C}{\delta^2}\int_0^T\int_{B_R(0)} |(\beta(u_s))_\varepsilon(x)-\beta(u_s(x))| h_\eta(s,x)^2u_s(x)\dx\ds\notag\\
	 	&\leq \frac{C}{\delta^2}||(\beta(u))_\varepsilon-\beta(u)||_{L^1((0,T)\times B_R(0))}\to 0 , \text{ as } \varepsilon \to 0.
	 \end{align}

	 \noindent \textbf{Regarding I:} \\
	 Let us define $G_\eta(r):=\int_0^r ((\beta\inv)(s^2)+\eta)^{-\zeta}\ds$, for all $r\in [0,\infty)$. Then, $G_\eta(r)$ is continuously differentiable in a sufficiently small neighbourhood of $[0,\infty)$ and, by the monotone convergence theorem, $\lim_{\eta\to 0}G_\eta(r) = G(r)$ for all $r\in [0,\infty)$. Choosing $K$ as above, we obtain similar to \eqref{theorem.MVSDE.PU:trick1} that for all $x,y \in \rd$ 
	 \begin{align}\label{theorem.MVSDE.PU:trick2}
		|G_\eta(\beta^\frac{1}{2}(u_s(x)))-G_\eta(\beta^\frac{1}{2}(u_s(y)))|
		 \leq (\bar{h}_\eta(x)\wedge \bar{h}_\eta(y))|\beta^\frac{1}{2}(u_s(x))-\beta^\frac{1}{2}(u_s(y))|,
	 \end{align}
	 where $\bar{h}_\eta(s,x):=\int_0^1(C_4 u_s(x)\theta^\frac{2}{m}+\eta)^{-\zeta}\mathrm d\theta, s \in [0,T], x\in \rd$. 
	 Note that $$\bar{h}_\eta(s,x)u_s(x) \leq C_4^{-\zeta}||(\cdot)^{-\frac{2\zeta}{m}}||_{L^1((0,1))}||u||^{1-\zeta}_{L^\infty}< \infty,$$ since $\frac{2\zeta}{m}\in [0,1)$.
	 Since $b\circ (G \circ \beta^\frac{1}{2})\inv \in \Lip_{loc}([0,\infty))$, we obtain, using \eqref{theorem.MVSDE.PU:trick2}, that
	\begin{align}\label{theorem.MVSDE.PU:I}
	 	\text{I}&\leq C\EE\int_0^{t\wedge\tau_R}\frac{ ||E||_{L^\infty}|b(u_s(X(s)))-b(u_s(Y(s)))| + ||b||_{L^\infty}|E(X(s))-E(Y(s))| }{|Z(s)|+\delta}\ds\notag\\
	 	&\leq C ||E||_{L^\infty} ||b\circ(G\circ \beta^\frac{1}{2})\inv||_{\Lip((G\circ \beta^\frac{1}{2})(K))}\notag\\
	 	& \ \ \ \ \  \cdot\liminf_{\eta\to 0}\EE\int_0^{t\wedge\tau_R} (\bar{h}_\eta(s,X(s))\wedge \bar{h}_\eta(s,Y(s)))\frac{|\beta^\frac{1}{2}(u_s(X(s)))-\beta^\frac{1}{2}(u_s(Y(s)))|}{|Z(s)|+\delta}\ds\notag\\
	 	&\ \ \ + C ||b||_{L^\infty} \EE\int_0^{t\wedge\tau_R}\frac{|E(X(s))-E(Y(s))|}{|Z(s)|+\delta}\ds
	 \end{align}
	 Arguing in a similar fashion as in \eqref{theorem.MVSDE.PU:II}-\eqref{theorem.MVSDE.PU:II.2}, the first integral on the right hand-side of \eqref{theorem.MVSDE.PU:I} can be estimated as
	 \begin{align*}
	 	\EE\int_0^{t\wedge\tau_R} (\bar{h}_\eta(s,X(s))\wedge \bar{h}_\eta(s,Y(s)))\frac{|\beta^\frac{1}{2}(u_s(X(s)))-\beta^\frac{1}{2}(u_s(Y(s)))|}{|Z(s)|+\delta}\ds
	 	&\leq C ||\nabla\beta^\frac{1}{2}(u)||_{L^2((0,T)\times\rd;\rd)}<\infty.
	 \end{align*}
	 
	Using an approximation of $E$ by convolution, we may again perform a similar argument as in \eqref{theorem.MVSDE.PU:II}-\eqref{theorem.MVSDE.PU:II.2} for the second integral on the right hand-side of \eqref{theorem.MVSDE.PU:I}. Indeed; setting $\bar{E}_\varepsilon:=E\ast \varphi_\varepsilon$, where the convolution is meant component-wise, we have
	\begin{align}\label{theorem.MVSDE.PU:I.11}
		\EE\int_0^{t\wedge\tau_R}\frac{|E(X(s))-E(Y(s))|}{|Z(s)|+\delta}\ds
		&=\EE\int_0^{t\wedge\tau_R}\frac{|\bar{E}_\varepsilon(X(s))-\bar{E}_\varepsilon(Y(s))|}{|Z(s)|+\delta}\ds
		+\EE\int_0^{t\wedge\tau_R}\frac{|\bar{E}_\varepsilon(X(s))-E(X(s))|}{|Z(s)|+\delta}\ds\notag\\
		&\ \ \   +\EE\int_0^{t\wedge\tau_R}\frac{|\bar{E}_\varepsilon(Y(s))-E(Y(s))|}{|Z(s)|+\delta}\ds.
	\end{align}
	For the first summand on the right hand-side we estimate similarly to \eqref{theorem.MVSDE.PU:II.1}
	\begin{align}\label{theorem.MVSDE.PU:E}
		\EE&\int_0^{t\wedge\tau_R}\frac{|\bar{E}_\varepsilon(X(s))-\bar{E}_\varepsilon(Y(s))|}{|Z(s)|+\delta}\ds
		\leq \EE \int_0^{t\wedge\tau_R} \frac{(i_{R+1})_\varepsilon(X(s)) + (i_{R+1})_\varepsilon(Y(s))}{|Z(s)|+\delta}\ds \notag\\
		&\leq 2\int_0^T \int_{B_R(0)}(i_{R+1})_\varepsilon(x) u_s(x)\dx\ds
		\leq 2T||i_{R+1}||_{L^1(B_R(0))}||u||_{L^\infty}<\infty.
	\end{align}
	The last two summands on the right hand-side in \eqref{theorem.MVSDE.PU:I.11} vanish with a similar reasoning as in \eqref{theorem.MVSDE.PU:II.2}.
	
All in all, we obtain by monotone convergence that
\begin{align}\label{theorem.MVSDE.PU:bound}
	\EE\ln\left(1+\sup_{\delta>0}\frac{|Z(t\wedge\tau_R)|^2}{\delta^2}\right)=\sup_{\delta>0}\EE\ln\left(1+\frac{|Z(t\wedge\tau_R)|^2}{\delta^2}\right) < \infty.
\end{align}
Necessarily, we therefore have for all $t\in [0,T]$
\begin{align}\label{theorem.MVSDE.PU:equality}
	X(t\wedge\tau_R) = Y(t\wedge\tau_R) \ \ \PP\text{-a.s.}
\end{align}
Note that $\sup_{t\in [0,T]}|X(t)|+|Y(t)| <\infty$ $\PP$-a.s. Hence, $\tau_R \to T$ $\PP$-a.s., as $R\to \infty$.
Therefore, taking the limit $R\to \infty$ in \eqref{theorem.MVSDE.PU:equality} and recalling that $X,Y \in C([0,T];\rd)$, we obtain that $\PP$-a.s.
\begin{align*}
	X(t) = Y(t) \ \ \forall t\in [0,T].
\end{align*}
This finishes the proof.
\end{proof}

\section{Proof of Theorem \ref{theorem.mainresult.2}}\label{section.theorem.mainresult.2.proof}
In this section, we prove the second main result, Theorem \ref{theorem.mainresult.2}.
Here, the corresponding Fokker--Planck equation is the classical porous medium equation with start in a Dirac measure. More precisely, we consider \eqref{FPKE} with $\beta(r) = |r|^{m-1}r, r\in \RR,$ for some $m>1$, $E\equiv b\equiv 0$, and $\nu = \delta_{x_0}$,
 i.e. 
\begin{align}\label{FPKE.PME}
	\partial_t u-\Delta(|u|^{m-1}u) = 0, \left.u\right|_{t=0}=\delta_{x_0},\ \ t\in [0,\infty),\tag{PME}
\end{align}
where $x_0\in\rd$.
In this case, \eqref{MVSDE} has the particular form
\begin{align}\label{MVSDE.PME}
	dX(t) \notag
	=&\ \sqrt{2\left|\frac{d\law{X(t)}}{dx}(X(t))\right|^{m-1}}\mathbbm{1}_{d\times d}\ \mathrm dW(t), \ \ t\in [0,T],\notag \\
	X(0)=&\ x_0.\tag{MVSDE.PME}
\end{align}

This section is split into three subsections.
In Section \ref{subsection.theorem.mainresult.2.proof:PME}, we recall the Barenblatt solution to \eqref{FPKE.PME} and recall and prove regularity results for powers of this special solution.
In Sections \ref{subsection.theorem.mainresult.2.proof:ingredient1} and \ref{subsection.theorem.mainresult.2.proof:ingredient2}, we then show that the conditions (i) and (ii) of the restricted Yamada--Watanabe, see Theorem \ref{theorem.restrictedYamadaWatanabe}, are fulfilled, respectively.
While the results in Sections 4.1 and 4.2 are true in the multi-dimensional case, we need to restrict ourselves to the one-dimensional setting in Section 4.3; more details can be found in Remark \ref{theorem.MVSDE.PU.PME:remark}.
\begin{proof}[Proof of Theorem \ref{theorem.mainresult.2}]
	The assertion follows directly Theorem \ref{theorem.MVSDE.PME.existence} and Theorem \ref{theorem.MVSDE.PME.PU} via Theorem \ref{theorem.restrictedYamadaWatanabe}.
\end{proof}
\subsection{The classical probability solution to \eqref{FPKE.PME}}\label{subsection.theorem.mainresult.2.proof:PME}
It is well known that $u$, as defined by the following formula, is a probability solution to \eqref{FPKE.PME}. It is called the \textit{Barenblatt solution} and has the form
\begin{align}\label{FPKE.PME.BarenblattPattleSolution}\tag{BB}
	u(t,x):= t^{-\alpha}(C - k|(x-x_0)t^{-\beta}|^2)^{\frac{1}{m-1}}_+,\ \ (t,x)\in (0,\infty)\times  \rd,
\end{align}
where $m>1, \alpha:=\frac{d}{d(m-1)+2}, k :=\frac{\alpha(m-1)}{2md}, \beta:=\frac{\alpha}{d}$, and $C\in (0,\infty)$ is a normalising constant.
In this section, $\beta$ shall not be confused with the diffusivity function from \eqref{MVSDE}.
By the argument presented in \cite[Section 3.1.]{gess2020timespace} (and as repeated in the proof of Proposition \ref{lemma.FPKE.PME.solution.regularity} below), it is simply not true that $\nabla u^\frac{m}{2} \in L^2_{(loc)}((0,T)\times\rd)$.
Hence, we cannot perform the same proof as for Theorem \ref{theorem.MVSDE.PU} in order to show that pathwise uniqueness holds among all weak solutions to \eqref{MVSDE.PME} with one-dimensional time marginal law densities $u$. However, it is true that $\nabla^{s} u^{\frac{m}{2}} \in L^2((0,T)\times\rd)$ for all $s\in (0,1)$. This will turn out to be enough to prove such a pathwise uniqueness statement in the case $d=1$, see Section \ref{subsection.theorem.mainresult.2.proof:ingredient2}.

The following proposition provides, separately, a necessary and a sufficient condition for fractional Sobolev regularity in space and integrability in time of certain powers of the Barenblatt solution. 
Proposition \ref{lemma.FPKE.PME.solution.regularity} \ref{lemma.FPKE.PME.solution.regularity:i} is taken from \cite[Section 3.1.]{gess2020timespace}. The proof of Theorem \ref{lemma.FPKE.PME.solution.regularity} \ref{lemma.FPKE.PME.solution.regularity:ii} uses and recalls the proof of the latter.
\begin{proposition}\label{lemma.FPKE.PME.solution.regularity}
	Let $m>1$, $p \in (0,m]$ and let $x_0 \in \rd$. Let $u$ denote the Barenblatt solution to \eqref{FPKE.PME} with $\left.u\right|_{t=0}=\delta_{x_0}$, i.e. $u$ is given by \eqref{FPKE.PME.BarenblattPattleSolution}. 
	\begin{enumerate}[label=(\roman*)]
		\item \label{lemma.FPKE.PME.solution.regularity:i} (\cite[Section 3.1.]{gess2020timespace})
		Assume that for some $s\in (0,\infty)$
		\begin{align}\label{lemma.FPKE.PME.solution.regularity:regularity}
		u^{p} \in L^{\frac{m}{p}}((0,T);\dot{W}^{s,\frac{m}{p}}(\rd)).
	\end{align}
	Then, $s< \frac{2p}{m}$.
		\item \label{lemma.FPKE.PME.solution.regularity:ii}
		Assume that $m,p$ fulfill $m(2p-m+1)>p$ and let $s\in (0,1)$ such that $s<\frac{2p}{m}$. Then,
		\begin{align*}
			u^{p} \in L^{\frac{m}{p}}((0,T);{W}^{s,\frac{m}{p}}(\rd)).
		\end{align*}
	\end{enumerate}
\end{proposition}
\begin{proof}
Let us recall the reasoning from \cite[Section 3.1.]{gess2020timespace}.  We set $F(x):=(C-k|x|^2)_+^\frac{p}{m-1}, x\in \rd$. Using the transformation rule for the Lebesgue integral, a direct computation reveals that for $s\in (0,1)$
	\begin{align}\label{lemma.FPKE.PME.solution.regularity:0}
		||u^p||_{L^{\frac{m}{p}}((0,T);\dot{W}^{s,\frac{m}{p}}(\rd))}^{\frac{m}{p}} = ||t^{-\alpha (m-1) - \beta\frac{sm}{p}}||_{L^1((0,T))}||F||_{\dot{W}^{s,\frac{m}{p}}(\rd)}^{\frac{m}{p}}.
	\end{align}
	The first factor on the right hand-side is finite as
	\begin{align*}
		-\alpha (m-1) - \beta\frac{sm}{p} > -1 \iff s < \frac{2p}{m}.
	\end{align*}
	A similar argument can be performed for $s \in (1,2)$ 
	since $\partial_{x_i}u^p(t,x)= t^{-\alpha m -\beta}\partial_{x_i}F(xt^{-\beta})$, for all $t>0, x\in \rd$ and $i \in \{1,...,d\}$.
	Hence, by induction, the proof of \ref{lemma.FPKE.PME.solution.regularity:i} is complete.
	
	For \ref{lemma.FPKE.PME.solution.regularity:ii}, let $m,p$ such that $m(2p-m+1)>p$ and let $0<s<\frac{2p}{m}\wedge 1$.
	Using the transformation rule for the Lebesgue integral, we have
	\begin{align*}
		||u^p||_{L^{\frac{m}{p}}((0,T)\times\rd)}^{\frac{m}{p}} = ||t^{-\alpha (m-1)}||_{L^1((0,T))}||F||_{L^{\frac{m}{p}}(\rd)}^{\frac{m}{p}}.
	\end{align*}
	Since $-\alpha (m-1)>-1$ and $F \in C_c(\rd)$, we obtain that $u^p \in L^{\frac{m}{p}}((0,T)\times \rd)$.
	By \eqref{lemma.FPKE.PME.solution.regularity:0}, it is sufficient to prove that $\nabla F \in L^{\frac{m}{p}}(\rd)$, as then, via the Brezis--Mironescu interpolation theorem (see \cite[Corollary 5.1]{brezis2018sobolev}),
	\begin{align*}
		||F||_{\dot{W}^{s,\frac{m}{p}}(\rd)} \leq ||F||_{L^\frac{m}{p}(\rd)}^{(1-s)}||\nabla F||_{L^{\frac{m}{p}}(\rd)}^{s}< \infty.
	\end{align*}
	So, let us show that $\nabla F \in L^{\frac{m}{p}}(\rd)$. 
		Clearly, $F$ is weakly differentiable with
	\begin{align*}
		\nabla F(x) = \frac{p}{m-1}(C-k|x|^2)^{\frac{p-(m-1)}{m-1}}(-2kx)\mathbbm{1}_{B_{\sqrt{\nicefrac{C}{k}}}(0)}(x), \text{ for a.e. } x\in \rd.
	\end{align*}
	Setting $\rho_1 := \frac{m(p-(m-1))}{p(m-1)}$, we have
	\begin{align*}
		||\nabla F||_{L^{\frac{m}{p}}}^{\frac{m}{p}} = \left(\frac{2kp}{m-1}\right)^\frac{m}{p}\int_{B_{\sqrt{\nicefrac{C}{k}}}(0)}(C-k|x|^2)^{\rho_1}|x|^\frac{m}{p}\dx.
	\end{align*}
		Integrating over shells, we obtain for $\rho_2:=\frac{m}{p}+(d-1)$
	\begin{align*}
		\int_{B_{\sqrt{\nicefrac{C}{k}}}(0)}(C-k|x|^2)^{\rho_1}|x|^\frac{m}{p}\dx
			= C_d\int_0^{\sqrt{\nicefrac{C}{k}}}(C-kr^2)^{\rho_1}r^{\rho_2}\dr,
	\end{align*}
	where $C_d \in (0,\infty)$ is some real constant depending only on the dimension $d$.
	Using the transformation rule for the Lebesgue integral, we obtain
	\begin{align*}
		\int_0^{\sqrt{\nicefrac{C}{k}}}(C-kr^2)^{\rho_1}r^{\rho_2}\dr
		= \frac{1}{2k^{\frac{\rho_2}{2}}}\int_0^C r^{\rho_1} \left(C-r\right)^\frac{\rho_2}{2}\dr.
	\end{align*}
The integral on the right hand-side is finite, since $\rho_1 >-1$ and $\rho_2 >-2$ by our assumptions.
	This completes the proof. 
\end{proof}
\subsection{Ingredient 1: A $P^{(u_t)}_{\delta_{x_0}}$-weak solution to \eqref{MVSDE.PME}}\label{subsection.theorem.mainresult.2.proof:ingredient1}

In \cite[Chapter 5]{barbu2020solutions}, Barbu and R\"ockner solved a huge class of nonlinear Fokker--Planck equations of the type \eqref{FPKE} with start in a bounded measure, and then, via the superposition principle procedure from \cite[Section 2]{barbu2019nonlinear}, solved the corresponding McKean--Vlasov SDE \eqref{MVSDE} based on the constructed solution to \eqref{FPKE}.
The following theorem is implied by \cite[Theorem 6.1 (b)]{barbu2020solutions}.
\begin{theorem}[$P^{(u_t)}_{\delta_{x_0}}$-weak solution]\label{theorem.MVSDE.PME.existence}
	Let $x_0 \in \rd$. Then, there exists a $P^{(u_t)}_{\delta_{x_0}}$-weak solution $(X,W)$ to \eqref{MVSDE}
	where $u$ is the Barenblatt solution to \eqref{FPKE.PME} with $\left.u\right|_{t=0}=\delta_{x_0}$.
\end{theorem}
\subsection{Ingredient 2: $P^{(u_t)}_{\delta_{x_0}}$-pathwise uniqueness for \eqref{MVSDE.PME}}\label{subsection.theorem.mainresult.2.proof:ingredient2}
Let $u$ denote the Barenblatt solution, i.e. $u$ is given by \eqref{FPKE.PME.BarenblattPattleSolution}.
As emphasised already in Section \ref{subsection.theorem.mainresult.2.proof:PME}, we cannot use the same lines of proof as for Theorem \ref{theorem.MVSDE.PU} in order to prove $P^{(u_t)}_{\delta_{x_0}}$-pathwise uniqueness. The reason for this is that $u$ does not satisfy $\nabla u^\frac{m}{2} \in L^2_{loc}((0,T)\times\rd)$, but only $\nabla^s u^\frac{m}{2} \in L^2((0,T)\times\rd)$ for all $s\in (0,1)$. 

The proof of the following theorem is similar to the one of Theorem \ref{theorem.MVSDE.PU}, however, it additionally relies on the ideas of the proof of \cite[Theorem. 2.15 (i)]{champagnat2018} together with \cite[Lemma 3.5]{champagnat2018}; the latter is a variation of Lemma \ref{appendix.lemma.lipschitztypeestimate} for Sobolev functions of fractional order $\nicefrac{1}{2}$, see \eqref{theorem.MVSDE.PME.PU:estimate.Lipschitztype.fractional} below.
\begin{theorem}[$P^{(u_t)}_{\delta_{x_0}}$-pathwise uniqueness]\label{theorem.MVSDE.PME.PU}
Let $d=1$, $m > 1$, and $x_0\in \RR$.
Let $u$ denote the Barenblatt solution to the one-dimensional equation \eqref{FPKE.PME} with $\left.u\right|_{t=0}=\delta_{x_0}$ as in \eqref{FPKE.PME.BarenblattPattleSolution}.

Let $T\in (0,\infty)$. Suppose $(X,W), (Y,W)$ are two $P^{(u_t)}_{\delta_{x_0}}$-weak solutions to \eqref{MVSDE.PME} up to time $T$ on a common stochastic basis $(\Omega,\FF,\PP;(\FF_t)_{t\in [0,T]})$ with respect to the same standard $d$-dimensional  $(\FF_t)$-Brownian motion $(W(t))_{t\in [0,T]}$ such that $X(0)=Y(0)$ $\PP$-a.s.
Then, ${\sup_{t\in [0,T]} |X(t)-Y(t)| = 0\  \PP\text{-a.s.}}$
\end{theorem}
\begin{proof}
Let  $Z(t):=X(t)-Y(t), t\in [0,T]$, and let $\tau_R:=\inf\{t \in [0,T] : |X(t)|\vee |Y(t)|\geq R \}\wedge T$, for $R>0$. 
Let $t\in [0,T]$. Fix an arbitrary $\delta >0$.
Let $f_\delta(r) := |r|\ln\left(\frac{|r|^2}{\delta^2}+1\right), r\in \RR$. Then $f_{\delta} \in C^2(\RR)$ with $|f'_\delta(r)| \leq C\ln\left(1+ \frac{|r|^2}{\delta^2}\right)$ and $|f''_\delta(r)| \leq \frac{C}{|r|+\delta}$ for all $r\in \RR$ and a constant $C\in (0,\infty)$ independent of $\delta$ and $r$.
By It\^o's formula and Fatou's lemma, we estimate
\begin{align}\label{theorem.MVSDE.PME.PU:estimate1}
	\EE &|Z(t\wedge\tau_R)|\ln\left(\frac{|Z(t\wedge\tau_R)|^2}{\delta^2}+1\right)\notag\\
	&\leq  C \EE \int_0^{t\wedge\tau_R} \frac{\left|u_s^\frac{m-1}{2}(X(s))-u_s^\frac{m-1}{2}(Y(s))\right|^2}{|Z(s)|+\delta}\ds
	\leq C \liminf_{\eta\to 0}\EE \int_0^{t\wedge\tau_R} \frac{\left|\sqrt{\frac{u_s^{m}(x)}{u_s(x)+\eta}}-\sqrt{\frac{u_s^{m}(y)}{u_s(y)+\eta}}\right|^2}{|Z(s)|+\delta}\ds
\end{align}
A similar argument as in the proof of Theorem \ref{theorem.MVSDE.PU} yields that for all $(s,x) \in [0,T]\times\RR$
\begin{align}\label{theorem.MVSDE.PME.PU:estimate2}
	\left|\sqrt{\frac{u_s^{m}(X(s))}{u_s(X(s))+\eta}}-\sqrt{\frac{u_s^{m}(Y(s))}{u_s(Y(s))+\eta}}\right| \leq (h_\eta(s,x)\wedge h_\eta(s,y)) \left|u_s^{\frac{m}{2}}(x)-u_s^{\frac{m}{2}}(y)\right|,
\end{align}
where we define $h_\eta(s,x) := (1+\frac{1}{m})\int_0^1 (\theta^\frac{2}{m} u_s(x)+\eta)^{-\frac{1}{2}}\mathrm d\theta$. Note that $h^2_\eta(s,x) u_s(x) \leq C$ for all $(s,x) \in [0,T]\times\RR$ for some constant $C\in (0,\infty)$ independent of $s,x,\eta$.
We recall that by \cite[Lemma 3.2]{champagnat2018} and Proposition \ref{lemma.FPKE.PME.solution.regularity}, for all $t>0$
\begin{align}\label{theorem.MVSDE.PME.PU:estimate.Lipschitztype.fractional}
	|u^\frac{m}{2}_t(x)-u^\frac{m}{2}_t(y)| \leq \left(\M|\partial^\frac{1}{2}_xu^\frac{m}{2}_t|(x) + \M|\partial^\frac{1}{2}_x u^\frac{m}{2}_t|(y)\right)|x-y|^\frac{1}{2}
\end{align}
for a.e. $x,y \in \RR$.
Using \eqref{theorem.MVSDE.PME.PU:estimate2} and \eqref{theorem.MVSDE.PME.PU:estimate.Lipschitztype.fractional} and that $u^m \in L^1((0,T)\times\rd)$, we estimate the right hand-side of \eqref{theorem.MVSDE.PME.PU:estimate1} via a convolution argument similar to \eqref{theorem.MVSDE.PU:II}-\eqref{theorem.MVSDE.PU:II.1} in the proof of Theorem \ref{theorem.MVSDE.PU} and obtain that
\begin{align*}
	\liminf_{\eta\to 0}\EE \int_0^{t\wedge\tau_R} \frac{\left|\sqrt{\frac{u_s^{m}(x)}{u_s(x)+\eta}}-\sqrt{\frac{u_s^{m}(y)}{u_s(y)+\eta}}\right|^2}{|Z(s)|+\delta}\ds
	 &\leq C ||\partial^{\frac{1}{2}}_x u||_{L^2((0,T)\times\RR)}< \infty.
\end{align*}
From the above, we may therefore conclude that
\begin{align*}
	\sup_{\delta>0}\EE|Z(t\wedge\tau_R)|\ln\left(\frac{|Z(t\wedge\tau_R)|^2}{\delta^2}+1\right) = \EE|Z(t\wedge\tau_R)|\ln\left(\sup_{\delta>0}\frac{|Z(t\wedge\tau_R)|^2}{\delta^2}+1\right)<\infty.
\end{align*}
Necessarily, we therefore have for all $t\in [0,T]$
\begin{align}\label{theorem.MVSDE.PME.PU:equality}
	X(t\wedge\tau_R) = Y(t\wedge\tau_R) \ \ \PP\text{-a.s.}
\end{align}
Note that $\sup_{t\in [0,T]}|X(t)|+|Y(t)|<\infty\  \PP$-a.s.
Hence, $\tau_R \to T \ \PP$-a.s., as $R\to \infty$. Therefore, taking the limit $R\to \infty$ in \eqref{theorem.MVSDE.PME.PU:equality} and recalling that $X,Y \in C([0,T];\RR)$, we obtain that $\PP$-a.s.
\begin{align*}
	X(t) = Y(t) \ \ \forall t\in [0,T].
\end{align*}
This finishes the proof.
\end{proof}
\begin{remark}\label{theorem.MVSDE.PU.PME:remark}
	Note that in the proof of Theorem \ref{theorem.MVSDE.PME.PU} the fundamental estimates on the first- and second-order derivatives of $f_\delta$ only work in the one-dimensional setting, see \cite[Proof of Theorem 2.15., p.1533]{champagnat2018}.
\end{remark}
\appendix
\addcontentsline{toc}{section}{Appendix}
\section*{Appendix}

\section{Lipschitz-type estimates}\label{section.appendix.maximalfunction}

We recall the definition and properties of (local) Hardy--Littlewood maximal functions on the set of Radon measures on $\rd$, where the latter are denoted as $M_{loc}(\rd;\rd)$ in the following. Furthermore, if $\mu \in M_{loc}(\rd;\rd)$, then $|\mu|$ denotes its variation measure.

The upcoming definition and the two lemmata presented afterwards are essentially taken from \cite{crippa2008estimates}.
\begin{definition}[{\cite[Definition A.1]{crippa2008estimates}}]\label{appendix.definition.localmaximalfunction}
	Let $\mu \in M_{loc}(\rd;\RR^{n})$ and $R\in (0,\infty]$. We define the (local) maximal function as
	\begin{align*}
		\M_{R}|\mu|(x) := \sup_{0<r<R}\frac{1}{|B_R(0)|}\int_{B_r(x)}|\mu|(\dx), x\in\rd.
	\end{align*}
	In the case $R=\infty$ we set $\M:=\M_\infty$.
	Furthermore, if $\mu$ is of the form $\mu(dx)=f(x)\dx$, where $f\in L^1_{loc}(\rd;\RR^n)$, then we write $\M_{R}|f|:=\M_{R}|\mu|$.
\end{definition}
\begin{lemma}[{\cite[Lemma A.3]{crippa2008estimates}}]\label{appendix.lemma.lipschitztypeestimate}
	Let $R\in (0,\infty]$ and $f\in W^{1,1}_{loc}(\rd;\RR^n)$. Then there exists a constant $C_d>0$ depending only on the dimension $d$, and a set $N \in \BBBB(\rd)$ with $\lambda^d(N)=0$ such that for all $x,y \in N^\complement$ with $|x-y|\leq R$
	\begin{align}\label{appendix.lemma.lipschitztypeestimate:1}
		|f(x)-f(y)| \leq C_d\left(\M_R|D f|(x) + \M_R|D f|(y)\right)|x-y|,
	\end{align}
	where $Df=(\partial_{x_j}f^i)_{1\leq i\leq d,1\leq j\leq n}$ denotes the matrix of all weak derivatives of $f$.
\end{lemma}
\begin{lemma}[{\cite[Lemma A.2]{crippa2008estimates}}]\label{appendix.lemma.boundednesslocalmaximalfunction}
	Let $\mu \in \mathcal{M}_{loc}(\rd;\RR^{m\times n})$ and $R\in (0,\infty]$. Then for a.e. $x\in \rd$, $M_R|\mu|(x)<\infty$.
	
	Let  $p\in (1,\infty)$. Then there exists a constant $C_{d,p}>0$ such that for all $f\in L^p_{loc}(\rd;\RR^n)$ and all $\rho >0$
	\begin{align*}
		\int_{B_\rho(0)} (\M_R |f|(x))^p\dx  \leq C_{d,p} \int_{B_{\rho+R}(0)}	|f(x)|^p\dx.
	\end{align*}
	For $p=1$, the previous statement does not hold. However, for $p=1$ one has the following weak estimate: There exists a constant $C_d>0$ such that for all $f\in L^1_{loc}(\rd;\RR^n)$
	\begin{align*}
		\lambda^d\left(\left\{x \in B_{\rho}(0) : \M_R|f|(x) >\alpha\right\}\right) \leq \frac{C_d}{\alpha}\int_{B_{\rho+R}(0)}|f(x)|\dx.
	\end{align*}

\end{lemma}
\section{A concise recapitulation on nonlinear evolution equations in Banach spaces and the Crandall--Ligget theorem}\label{section.appendix.mildsolutionframework}
This section's content is borrowed from \cite[Chapter 4]{barbu2010NDE} and presented in a concise form. For the details, we refer to \cite{barbu2010NDE}.\\

\noindent\textbf{The problem.} Let $(X,||\cdot||)$ be a Banach space and let $A: D(A)\subset X \to X$ be a (possibly multi-valued) operator. Under which conditions on $\nu\in X$ may one find a (unique) solution to the evolution equation 
\begin{align}\label{appendix.crandallLigget.equation}\tag{EE}
	\partial_tu+A(u)=0,\ \ \left.u\right|_{t=0}=\nu,
\end{align}
and in what sense?
\\
\\
In the following $I$ will denote the identity map on $X$.
\begin{definition}[(m-)-accretivity]
	$(A,D(A))$ is called \textit{accretive}, if for all $i =1,2, \ (x_i,y_i) \in \mathrm{Graph}(A),$
	\begin{align}\label{accretivity}
		\norm{x_1-x_2} \leq \norm{x_1-x_2 + \lambda(y_1-y_2)}\ \ \forall \lambda>0.
	\end{align}
	If $(A,D(A))$ is accretive and $(I+\lambda A)(D(A)) = X$, then $(A,D(A))$ is called m-accretive.
\end{definition}
\begin{remark}
	Due to \eqref{accretivity}, the operator $(I+\lambda A)\inv$ is single-valued for all $\lambda >0$.
\end{remark}
From now on assume that $(A,D(A))$ is $m$-accretive.
The following definition provides a concept of solution to \eqref{appendix.crandallLigget.equation} based on a finite difference scheme.
\begin{definition}\label{appendix.crandallLigget.mildSolution}
	A continuous function $u:[0,\infty) \to X$ is called a mild solution to \eqref{appendix.crandallLigget.equation}, if for each $T>0$
	\begin{align*}
		u(t)=\lim_{h\to 0} u_h(t) \text{ in $X$, uniformly on $[0,T]$,}
	\end{align*}
	where $u_h:[0,T]\to X$ is a step function, defined by the finite difference scheme
	\begin{align*}
		u_h(t)=u_h^{i+1}\ \forall t\in (ih,(i+1)h], i=0,1,...,N-1, u_h(0)=\nu,\\
		\text{where } u_h^{i+1}+hA(u_h^{i+1}) = u_h^i,\ Nh=T,\\
		\text{and } u_h^0=\nu.
	\end{align*}
\end{definition}

The following theorem is a special case of the \textit{Crandall--Ligget theorem}.
\begin{theorem}[{cf. \cite[Theorem 4.3]{barbu2010NDE}}]\label{appendix.theorem.crandallLigget}
Let $(A,D(A))$ be m-accretive. Then, for all $\nu \in \overline{D(A)}$ (closure of $D(A)$ in $X$) there exits a unique mild solution to \eqref{appendix.crandallLigget.equation}.
\end{theorem}
\textbf{Acknowledgements:}
I am very grateful to Prof. Dr. Michael R\"ockner, with whom I shared valuable discussions.
Moreover, I gratefully acknowledge the support by the German Research Foundation (DFG) through the CRC 1283.
\printbibliography
\end{document}